\documentclass{amsart}
\usepackage{amssymb}
\usepackage{}
\usepackage{mathrsfs}
\usepackage{amsfonts}
\usepackage{amsfonts,amssymb,amsmath,amsthm}
\usepackage{url}
\usepackage{enumerate}

\newtheorem{theorem}{Theorem}[section]
\newtheorem{lemma}[theorem]{Lemma}
\newtheorem{proposition}[theorem]{Proposition}
\newtheorem{corollary}[theorem]{Corollary}
\theoremstyle{definition}

\newtheorem{remark}{Remark}
\newtheorem{example}{Example}
\numberwithin{equation}{section}

\DeclareMathOperator{\vol}{vol}
\DeclareMathOperator{\di}{div}

\author{Lixia Yuan}
\address{
School of Mathematics and Physics\\
Shanghai Normal University\\
Shanghai, China}
\email{yuanlixia@shnu.edu.cn}

\author{Wei Zhao*}
 \thanks{*  Corresponding author.}
\address{
Department of Mathematics\\
East China University of Science and Technology\\
Shanghai, China}
\email{szhao\underline{ }wei@yahoo.com}

\author{Yibing Shen}
\address{
School of Mathematics\\
Zhejiang University\\
Hangzhou, China}
\email{yibingshen@zju.edu.cn}

\begin{document}

\title[]{Improved Hardy and Rellich inequalities on nonreversible Finsler manifolds}

\begin{abstract}
In this paper, we study the sharp constants of quantitative Hardy and Rellich inequalities on nonreversible Finsler manifolds equipped with arbitrary measures.
In particular, these inequalities can be globally refined by adding remainder terms like the
Brezis-V\'azquez improvement, if Finsler manifolds are of strictly negative flag curvature, vanishing S-curvature and finite uniformity constant. Furthermore, these results remain valid when Finsler metrics are reversible.
\end{abstract}
\maketitle

\section{Introduction} \label{sect1}

Let $\Omega$ be either $\mathbb{R}^n$ or a bounded domain in $\mathbb{R}^n$ containing $0$. Then the Hardy and  Rellich inequalities can be stated as follows, respectively: for all $u\in C^\infty_0(\Omega)$,
\begin{align*}
\int_\Omega |\nabla u|^2 dx&\geq \frac{(n-2)^2}{4}\int_\Omega \frac{u^2}{|x|^2}dx,   \text{ if }n\geq 3,\tag{1.1}\label{old1.1}\\
\int_\Omega |\Delta u|^2 dx&\geq \frac{n^2(n-4)^2}{16}\int_\Omega \frac{u^2}{|x|^4}dx,   \text{ if }n\geq 5.\tag{1.2}\label{old1.2}
\end{align*}
Here both the constant $\frac{(n-2)^2}{4}$ and $\frac{n^2(n-4)^2}{16}$ are sharp but never archived, which inspires one to improve these inequalities by adding some nonnegative correction terms to the right-hand side of (\ref{old1.1}) and (\ref{old1.2}). In fact, Brezis-V\'azquez in \cite{BV} showed that if $\Omega$ is bounded, then there exists a constant $C_\Omega>0$ such that
\[
\int_\Omega |\nabla u|^2 dx\geq \frac{(n-2)^2}{4}\int_\Omega \frac{u^2}{|x|^2}dx+C_\Omega\int_\Omega u^2 dx,
\]
while Gazzola-Grunau-Mitidieri \cite{GGM} proved that there are positive constants $C_1$ and $C_2$ so that
\[
\int_\Omega |\Delta u|^2 dx\geq \frac{n^2(n-4)^2}{16}\int_\Omega \frac{u^2}{|x|^4}dx+C_1\int_\Omega \frac{u^2}{|x|^2}dx+C_2\int_\Omega u^2 dx.
\]
We refer to \cite{ACP,AFT,BFT,BV,GGM,VZ} and references therein for more improvements of (\ref{old1.1}) and (\ref{old1.2}).

Hardy and Rellich inequalities have also been investigated in the Riemannian setting. Carron \cite{Ca} obtained some weighted $L^2$ Hardy inequalities on complete, non-compact Riemannian manifolds. Afterwards, Kombe-\"Ozaydin \cite{KO} and Yand-Su-Kong \cite{YSK} established some weighted $L^p$ Hardy and Rellich inequalities and presented various Brezis-Vazquez type improvements on complete Riemannian manifolds.

Finsler geometry is Riemannian geometry without quadratic restriction. In order to present the corresponding results in Finsler geometry, we introduce some notations and notions first.

Let $(M,F)$ be a Finsler manifold. The Finsler metric $F$ is called reversible if $F(x,-y)=F(x,y)$, otherwise it is called nonreversible (or irreversible). Clearly, a Riemannian metric is always a reversible Finsler metric. Let $d_F:M\times M\rightarrow \mathbb{R}$ denote the distance function induced by $F$. It is remarkable that $d_F$ is usually asymmetric, that is, $d_F(p,q)\neq d_F(q,p)$. Thus, given a point $p\in M$, we define forward and backward distance functions (with respect to $p$) as follows:
\[
\rho_+(x):=d_F(p, x),\ \rho_-(x):=d_F( x,p), \ x\in M.
\]
For a reversible metric, $\rho_+$ coincides with $\rho_-$ and hence, we use $\rho(x)$ to denote it.

As far as we know, Krist\'aly-Repov$\breve{\text{s}}$ \cite{KR} first considered  weighted Hardy and Rellich inequalities in the Finsler setting. More precisely, let $(M,F)$ be an $n$-dimensional reversible Finsler-Hadamard manifolds, equipped with the Busemann-Hausdorff measure
$d\mathfrak{m}_{BH}$, with nonpositive flag curvature $\mathbf{K}\leq k\leq 0$ and vanishing S-curvature of $d\mathfrak{m}_{BH}$. Then the Hardy inequality in \cite{KR} can be  stated as follows: for any $\beta\in \mathbb{R}$ with $n-2>\beta$ and all $u\in C^\infty_0(M)$,
\begin{align*}
\int_M \frac{F^2(\nabla u(x))}{\rho^\beta(x)}\,d \mathfrak{m}_{BH}(x)\geq& \frac{(n-2-\beta)^2}4\int_M \frac{u^2(x)}{\rho^{\beta+2}(x)}\,d \mathfrak{m}_{BH}(x)\\
+&\frac{(n-1)(n-\beta-2)}{2}\int_M\frac{u^2(x)}{\rho^{\beta+2}(x)} D_k(\rho(x))\,d \mathfrak{m}_{BH}(x),\tag{1.3}\label{new1.333}
\end{align*}
where
\begin{align*}
D_k(t):=\left\{
\begin{array}{lll}
& 0, & \ \ \ \text{if } k=0, \\
\\
&\sqrt{|k|}\,t\coth(\sqrt{|k|} \, t)-1, & \ \ \ \text{if } k<0.
\end{array}
\right.
\end{align*}
Furthermore, the Rellich inequality on reversible Finsler-Hadamard manifolds in \cite{KR}  reads, for any $\beta\in \mathbb{R}$ with $-2<\beta<n-4$ and every $u\in C^\infty_0(M)$ with $G^\beta_F(u)=0$,
\begin{align*}
\int_M \frac{(\Delta u(x))^2}{\rho^\beta(x)}\,d \mathfrak{m}_{BH}(x)\geq& \frac{(n+\beta)^2(n-4-\beta)^2}{16}\int_M \frac{u^2(x)}{\rho^{\beta+4}(x)}\,d \mathfrak{m}_{BH}(x)\\
+&\frac{(n-1)(n-2)(n+\beta)(n-\beta-4)}{4}\int_M\frac{u^2(x)}{\rho^{\beta+4}(x)} D_k(\rho(x))\,d \mathfrak{m}_{BH}(x),\tag{1.4}\label{new1.444}
\end{align*}
where
\[
G^\beta_F(u):=\int_M \left[u^2(x) \Delta (\rho^{-\beta-2}(x) )- \rho^{-\beta-2}(x) \Delta( u^2(x)) \right]\,d \mathfrak{m}_{BH}(x).
\]
According to \cite{FKV,KR}, the constants $\frac{(n-2-\beta)^2}4$ and $\frac{(n+\beta)^2(n-4-\beta)^2}{16}$  in (\ref{new1.333}) and (\ref{new1.444})   are still sharp but never achieved. So it is natural to study whether these two inequalities can be improved by adding nonnegative items as before.

Also note that there are infinitely many nonreversible Finsler metrics on a manifold.
For example, a Randers metric $F=\alpha+\beta$ is irreversible, where $\alpha$ is
a Riemannian metric and $\beta$ is a $1$-form with $\|\beta\|_\alpha<1$. For a nonreversible Finsler metric $F$, one can use the reversibility $\lambda_F$ \cite{R} and the  uniformity constant $\Lambda_F$ \cite{E} to study its asymmetry, i.e.,
\[
\lambda_F:=\sup_{X\in TM\setminus\{0\}}\frac{F(-X)}{F(X)},\  \Lambda_F:=\sup_{X,Y,Z\in TM\setminus\{0\}}\frac{g_X(Y,Y)}{g_Z(Y,Y)},
\]
where $g$ is the fundamental tensor induced by $F$. It is easy to see that $1\leq \lambda_F\leq \sqrt{\Lambda_F}$. In particular, $\lambda_F=1$ if and only if $F$ is reversible while $\Lambda_F=1$ if and only if $F$ is Riemannian.
One may utilize these two non-Riemannian geometric quantities to study Hardy
and Rellich inequalities on nonreversible Finsler manifolds, but it is usually hard to discuss the sharpness of the constants by this method.
\begin{example}[{\cite[Theorem 4.1]{FKV}}]
Let $(M,F)$ be a nonreversible Finsler-Hadamard manifold $(M,F)$ with  vanishing S-curvature of $d\mathfrak{m}_{BH}$. Suppose that $\Omega$ is an open domain containing $p$. Then for all $u\in C^\infty_0(\Omega)$,
\[
\int_\Omega F^2(\nabla u)\,d\mathfrak{m}_{BH}(x)\geq \mu\int_\Omega\frac{u^2(x)}{\rho_+^2(x)}d\mathfrak{m}_{BH}(x),
\]
where  $0\leq \mu< \frac{(n-2)^2}{4\Lambda_F\lambda^2_F}$.
\end{example}

On the other hand, unlike the Riemannian case, the
measures on a Finsler manifold can be defined in various ways \cite{BBI} and essentially different results
may be obtained. Besides the Busemann-Hausdorff measure, another measure used frequently is so-called the Holmes-Thompson measure, whose properties are much different from those of the Busemann-Hausdorff one, cf. \cite{AlB,AlT}. For instance, these two measures of a Randers metric $F=\alpha+\beta$ are as follows:
\[
d\mathfrak{m}_{BH}=\left(1-\|\beta\|_\alpha^2\right)^\frac{n+1}2 dV_\alpha,\ d\mathfrak{m}_{HT}=dV_\alpha,\tag{1.5}\label{1.3**}
\]
where $dV_\alpha$ is the Riemannian measure induced by $\alpha$. Even in the reversible case, $d\mathfrak{m}_{BH}\leq d \mathfrak{m}_{HT}$ with equality if and only if $F$ is Riemannian. So one may ask that whether the Hardy and Rellich inequalities above hold for the Holmes-Thompson measure or other more general measures. However,
there is less
literature for this subject.

The first goal of this paper is
to study analogues of Hardy and Rellich inequalities on irreversible Finsler manifolds which are valid for all measures.  In particular, the sharpness of constants is discussed. Moreover, the second goal of this paper is to show that there exist global Brezis-V\'azquez improvements of Hardy and Rellich inequalities on general (i.e., both reversible and nonreversible) Finsler spaces.

 Before presenting our main results, we introduce some notations.
Let $(M,F)$ be an $n$-dimensional Finsler manifold as before.
Given $u\in C^\infty_0(M)$, define a measurable function $\rho_{u}$ by
\begin{align*}
\rho_{u}(x):=\left\{
\begin{array}{lll}
& \rho_-(x), & \ \ \ \text{if } u(x)>0, \\
\\
& \rho_+(x), & \ \ \ \text{if } u(x)<0,\\
\\
&\frac12\left[\rho_+(x)+\rho_-(x)\right], & \ \ \ \text{if } u(x)=0.
\end{array}
\right.
\end{align*}
It is not hard to see that $\rho_{u}$ is differentiable almost everywhere and especially, $\rho_{u}=\rho$ for all $u\in C^\infty_0(M)$ in the reversible case. Now we have the following Hardy inequalities.
\begin{theorem}\label{Th11}Let $(M,F)$ be an $n$-dimensional complete non-reversible Finsler manifold equipped with an arbitrary measure $d\mathfrak{m}$. Suppose that   the S-curvature of $d\mathfrak{m}$ vanishes. Given any $\beta\in \mathbb{R}$ with $n-2>\beta$.

(i) If the flag curvature $\mathbf{K}\leq k\leq 0$, then  all $u\in C^\infty_0(M)$,
\begin{align*}
\int_M  \frac{F^{2}(\nabla u)}{\rho^\beta_u(x)}d\mathfrak{m}(x)&\geq\frac{(n-2-\beta)^2}{4} \int_M \frac{u^2(x)}{{\rho^{\beta+2}_u}(x)}d\mathfrak{m}(x)\\
&+\frac{(n-1)(n-2-\beta)}{2}\int_M  \frac{u^2(x)}{{\rho^{\beta+2}_u}(x)}\cdot D_{k}(\rho_u(x))d\mathfrak{m}(x),
\end{align*}
where the constant $\frac{(n-2-\beta)^2}{4} $ is sharp.

(ii) If $\mathbf{K}\leq k<0$ and $\Lambda_F<\infty$, then there exist a constant $C=C(n,\Lambda_F,k)>0$ such that for  all $u\in C^\infty_0(M)$,
\begin{align*}
\int_M  \frac{F^{2}(\nabla u)}{\rho^\beta_u(x)}d\mathfrak{m}(x)&\geq\frac{(n-2-\beta)^2}{4} \int_M \frac{u^2(x)}{{\rho^{\beta+2}_u}(x)}d\mathfrak{m}(x)\\
&+\frac{(n-1)(n-2-\beta)}{2}\int_M  \frac{u^2(x)}{{\rho^{\beta+2}_u}(x)}\cdot D_{k}(\rho_u(x)) d\mathfrak{m}(x)\\
&+\frac{C}{\Lambda_F}\int_M\frac{u^2(x)}{\rho^{\beta}_u(x)}d\mathfrak{m}(x),
\end{align*}
where the constant $\frac{(n-2-\beta)^2}{4}$ is sharp.
\end{theorem}

Now we turn to present our results concerning Rellich inequalities. Given $u\in C^\infty_0(M)$, set
\[
G^\beta_{F,d\mathfrak{m}}(u):=\int_M\left[ u^2(x)\cdot \varrho_{u,\beta}(x)+2\rho^{-\beta-2}_u(x)\cdot\di(u\nabla u)  \right]d\mathfrak{m}(x),
\]
where
\begin{align*}
\varrho_{u,\beta}(x):=\left\{
\begin{array}{lll}
& -\Delta(\rho^{-\beta-2}_-)(x), & \ \ \ \text{if } u(x)>0, \\
\\
& \Delta(-\rho^{-\beta-2}_+)(x), & \ \ \ \text{if } u(x)<0,\\
\\
&\frac12\left[-\Delta(\rho^{-\beta-2}_-)(x) +\Delta(-\rho^{-\beta-2}_+)(x)\right], & \ \ \ \text{if } u(x)=0.
\end{array}
\right.
\end{align*}
One can show that
\[
C^\infty_{0,F,d\mathfrak{m},\beta}(M):=\{u\in C^\infty_0(M):\,G^\beta_{F,d\mathfrak{m}}(u)=0\}.
 \]
 is not empty (see Proposition \ref{nonempty} below). In particular, if $F$ is reversible and $d\mathfrak{m}$ is the Busemann-Hausdorff measure, then $G^\beta_{F,d\mathfrak{m}}(u)$ is exactly $G^\beta_F(u)$. Furthermore, $C^\infty_{0,F,d\mathfrak{m},\beta}(M)=C^\infty_0(M)$ in the Riemannian case. Now we have the following Rellich inequalities.
\begin{theorem}\label{Th22}Let $(M,F)$ be an $n$-dimensional complete non-reversible Finsler manifold equipped with an arbitrary measure $d\mathfrak{m}$. Suppose that  the S-curvature of $d\mathfrak{m}$ vanishes.

(i) If the flag curvature $\mathbf{K}\leq k\leq 0$, then for any $\beta\in \mathbb{R}$ with $-2<\beta<n-4$ and
each $u\in C^\infty_{0,F,d\mathfrak{m},\beta}(M)$,
\begin{align*}
\int_M\frac{(\Delta u)^2}{\rho^\beta_u(x)} d\mathfrak{m}(x)&\geq \frac{(n+\beta)^2(n-4-\beta)^2}{16}\int_M \frac{u^2(x)}{{\rho^{\beta+4}_u}(x)}d\mathfrak{m}(x)\\
&+\frac{(n-1)(n-2)(n+\beta)(n-4-\beta)}{4}\int_M  \frac{u^2(x)}{{\rho^{\beta+4}_u}(x)}\cdot D_{k}(\rho_u(x))d\mathfrak{m}(x),
\end{align*}
where the constant $\frac{(n+\beta)^2(n-4-\beta)^2}{16}$ is sharp.

(ii) If $\mathbf{K}\leq k<0$ and $\Lambda_F<\infty$, then there exist a constant $C=C(n,\Lambda_F,k)>0$ such that for any $\beta\in \mathbb{R}$ with $0\leq \beta<n-2$ and each $u\in C^\infty_{0,F,d\mathfrak{m},\beta}(M)$,
\begin{align*}
\int_M \frac{(\Delta u)^2}{\rho^{\beta}_u(x)}d\mathfrak{m}(x)\geq&\frac{(n+\beta)^2(n-\beta-4)^2}{16}\int_M\frac{u^2(x)}{\rho^{\beta+4}_u(x)}d\mathfrak{m}(x)\\
+&\frac{(n-1)(n-2)(n+\beta)(n-\beta-4)}{4}\int_M \frac{u^2(x)}{\rho^{\beta+4}_u(x)}D_{k}(\rho_u)d\mathfrak{m}(x)\\
+&\frac{(n-2-\beta)(n-2+\beta)C}{2\Lambda_F}\int_M\frac{u^2(x)}{\rho^{\beta+2}_u(x)}d\mathfrak{m}(x)\\
+&\frac{(n-1)(n-2)C}{\Lambda_F}\int_M \frac{u^2(x)}{\rho^{\beta+2}_u(x)}D_{k}(\rho_u)d\mathfrak{m}(x)+\frac{C^2}{\Lambda_F^2}\int_M\frac{u^2(x)}{\rho^{\beta}_u(x)}d\mathfrak{m}(x),
\end{align*}
where the constant $\frac{(n+\beta)^2(n-4-\beta)^2}{16}$ is sharp.
\end{theorem}

We remark here that the manifolds $M$ in Theorem \ref{Th11} and Theorem \ref{Th22} are not necessarily noncompact or simply connected. In fact, all the results remain true if $M$ is a (connected) closed manifold or a (connected) domain containing $p$. Besides, both Theorem \ref{Th11} and Theorem \ref{Th22} remain valid in the reversible case.

The main difficulty here is to deal with the gradient and laplacian of $\rho_-$. Inspired by \cite{BCS,Ot}, our idea is to take advantage of the reverse Finsler metric. It is also remarkable that this method can be employed to solve many other problems concerning nonreversible Finsler manifolds. See Remark \ref{3remark} for instance.  Moreover, we establish a refined inequality about Finsler metrics (Theorem \ref{ineq}), which is genuinely new  and can be viewed as an extension of the  Cauchy-Schwarz inequality in the Finsler setting.

The paper is organized as follows. Section 2 is devoted to preliminaries of Finsler geometry. Theorem \ref{Th11} and Theorem \ref{Th22} are proved in Section 3 and 4, respectively. An example is given in Section 5.

\section{Preliminaries}
In this section, we recall some definitions and properties about Finsler manifolds. See \cite{BCS,Sh1,Sh3} for more details.

Let $M$ be a (connected) $n$-dimensional manifold equipped with a Finsler metric $%
F:TM\rightarrow [0,\infty)$.
Let $(x,y)=(x^i,y^i)$ be a local coordinate system on $%
TM$. Define
\begin{align*}
g_{ij}(x,y):=\frac12\frac{\partial^2 F^2(x,y)}{\partial y^i\partial y^j}, \ G^i(y):=\frac14 g^{il}(y)\left\{2\frac{\partial g_{jl}}{\partial x^k}(y)-\frac{\partial g_{jk}}{\partial x^l}(y)\right\}y^jy^k,
\end{align*}
where $G^i$ is the so-called geodesic coefficient. A smooth curve $\gamma(t)$ in $M$ is called a (constant speed) geodesic if it satisfies
\[
\frac{d^2\gamma^i}{dt^2}+2G^i\left(\frac{d\gamma}{dt}\right)=0.
\]
The Riemannian curvature $R_y$ of $F$ is a family of linear transformations on tangent spaces. More precisely, set
$R_y:=R^i_k(y)\frac{\partial}{\partial x^i}\otimes dx^k$, where
\[
R^i_{\,k}(y):=2\frac{\partial G^i}{\partial x^k}-y^j\frac{\partial^2G^i}{\partial x^j\partial y^k}+2G^j\frac{\partial^2 G^i}{\partial y^j \partial y^k}-\frac{\partial G^i}{\partial y^j}\frac{\partial G^j}{\partial y^k}.
\]
Let $P:=\text{Span}\{y,v\}\subset T_xM$ be a plane. The flag curvature is defined by
\[
\mathbf{K}(P,y):=\frac{g_y\left( R_y(v),v  \right)}{g_y(y,y)g_y(v,v)-g^2_y(y,v)}.
\]

Set $S_xM:=\{y\in T_xM:F(x,y)=1\}$ and $SM:=\cup_{x\in M}S_xM$. The reversibility $\lambda_F$ and the uniformity constant $\Lambda_F$ of $(M,F)$ are defined as follows:
\[
\lambda_F:=\underset{y\in SM}{\sup}F(-y),\ \Lambda_F:=\underset{X,Y,Z\in SM}{\sup}\frac{g_X(Y,Y)}{g_Z(Y,Y)}.
\]
Clearly, ${\Lambda_F}\geq \lambda_F^2\geq 1$. It is easy to see that $\lambda_F=1$ if and only if $F$ is reversible, while $\Lambda_F=1$ if and
only if $F$ is Riemannian.

The average Riemannian metric $\hat{g}$ induced by $F$ is defined by
\[
\hat{g}(X,Y)=\frac{1}{\vol(x)}\int_{S_xM}g_y(X,Y)d\nu_x(y),\ \forall\,X,Y\in T_xM,
\]
where $\vol(x)=\int_{S_xM}d\nu_x(y)$, and $d\nu_x$ is the Riemannian volume form of $S_xM$ induced by $F$.
It is noticeable that
\[
\Lambda^{-1}_F\cdot F^2(X)\leq \hat{g}(X,X)\leq \Lambda_F \cdot F^2(X).\tag{2.1}\label{2.000}
\]

The dual Finsler metric $F^*$ on $M$ is
defined by
\begin{equation*}
F^*(\eta):=\underset{X\in T_xM\backslash \{0\}}{\sup}\frac{\eta(X)}{F(X)}, \ \
\forall \eta\in T_x^*M,
\end{equation*}
which is also a Finsler metric on $T^*M$.
The Legendre transformation $\mathfrak{L} : TM \rightarrow T^*M$ is defined
by
\begin{equation*}
\mathfrak{L}(X):=\left \{
\begin{array}{lll}
& g_X(X,\cdot) & \ \ \ X\neq0, \\
& 0 & \ \ \ X=0.%
\end{array}
\right.
\end{equation*}
In particular, $F^*(\mathfrak{L}(X))=F(X)$. Now let $f : M \rightarrow \mathbb{R}$ be a smooth function on $M$. The
gradient of $f$ is defined as $\nabla f = \mathfrak{L}^{-1}(df)$. Thus,  $df(X) = g_{\nabla f} (\nabla f,X)$.

Let $\sigma:[0,1]\rightarrow M$ be a Lipschitz continuity path from $p,q$. The length of $\sigma$ is defined by
\[
L_F(\sigma):=\int^1_0 F(\dot{\sigma}(t))dt.
\]
Define the distance function $d_F:M\times M\rightarrow [0,+\infty)$ by
$d_F(p,q):=\inf L_F(\sigma)$,
where the infimum is taken over all
Lipshitz continuous paths $\gamma:[a,b]\rightarrow M$ with
$\gamma(a)=p$ and $\gamma(b)=q$. It should be remarked that  $d_F(p,q)\neq d_F(q,p)$ generally unless $F$ is reversible.

The forward and backward metric balls $B^+_p(R)$ and $B^-_p(R)$ are defined by
\[
B^+_p(R):=\{x\in M:\, d_F(p,x)<R\},\ B^-_p(R):=\{x\in M:\, d_F(x,p)<R\}.
\]
If $F$ is reversible, forward metric balls coincide with backward ones.

Given $p\in M$, set $\rho_+(x):=d_F(p,x)$ and $\rho_-(x):=d_F(x,p)$. In general, $\rho_+(x)\neq\rho_-(x)$ unless $F(p,\cdot)$ is reversible, cf. \cite[Exercise 6.3.4]{BCS}. Moreover, it follows from \cite[Lemma 3.2.3]{Sh1} that
\[
F(\nabla \rho_+)=1,\ F(\nabla (-\rho_-))=1, \tag{2.2}\label{1.1}
\]
hold almost everywhere. In the reversible case, we just use $\rho(x)$ to denote $d_F(p,x)$.

Since the metric $d_F$ is usually not symmetric, forward and backward Cauchy sequences are introduced in order to study the completeness of a Finsler manifold.
A sequence $\{x_i\}$ in $M$ is called a forward (resp., backward) Cauchy sequence if, for all $\epsilon>0$, there exists a positive integer $N=N(\epsilon)$ such that when $j>i\geq N$, $d(x_i,x_j)<\epsilon$ (resp., $d(x_j,x_i)<\epsilon$).

A Finsler manifold $(M,F)$ is said to be forward  (resp., backward) complete if every forward (resp., backward) Cauchy sequence converges in $M$. In particular, $(M,F)$ is called complete if it is both forward and backward complete.
The Hopf-Rinow theorem (\cite[Theorem 6.6.1]{BCS}) yields that if $(M,F)$ is forward complete if and only if it is forward geodesically complete, that is, every geodesic $\gamma(t)$, $a\leq t<b$, can be extended to a geodesic defined on $a\leq t<\infty$. Similarly,  $(M,F)$ is  backward complete if and only if it is backward geodesically complete, that is, every geodesic $\gamma(t)$, $a< t\leq b$, can be extended to a geodesic defined on $-\infty< t\leq b$.

Let $d\mathfrak{m}$ be a measure on $M$. In a local coordinate system $(x^i)$,
express $d\mathfrak{m}=\sigma(x)dx^1\wedge\cdots\wedge dx^n$. In particular,
the Busemann-Hausdorff measure $d\mathfrak{m}_{BH}$ and the Holmes-Thompson measure $d\mathfrak{m}_{HT}$ are defined by
\begin{align*}
&d\mathfrak{m}_{BH}:=\frac{\vol(\mathbb{B}^{n})}{\vol(B_xM)}dx^1\wedge\cdots\wedge dx^n,\\
 &d\mathfrak{m}_{HT}:=\left(\frac1{\vol(\mathbb{B}^{n})}\int_{B_xM}\det g_{ij}(x,y)dy^1\wedge\cdots\wedge dy^n \right) dx^1\wedge\cdots\wedge dx^n,
\end{align*}
where $B_xM:=\{y\in T_xM: F(x,y)<1\}$.
Define the distortion of $(M,F,d\mathfrak{m})$ as
\begin{equation*}
\tau(y):=\log \frac{\sqrt{\det g_{ij}(x,y)}}{\sigma(x)}, \text{ for $y\in T_xM\backslash\{0\}$}.
\end{equation*}
And the S-curvature $\mathbf{S}$ is defined by
\begin{equation*}
\mathbf{S}(y):=\frac{d}{dt}[\tau(\dot{\gamma}(t))]|_{t=0},
\end{equation*}
where $\gamma(t)$ is the geodesic with $\dot{\gamma}(0)=y$. Given $h\geq 0$, we say $|\mathbf{S}|\leq (n-1)h$ if
\[
-(n-1)h\cdot F(y)\leq \mathbf{S}(y)\leq (n-1)h  \cdot F(y),\ \forall y\in TM-\{0\}.
\]

Given $y\in S_pM$, let $\gamma_y(t)$, $t\geq 0$ denote the geodesic with $\dot{\gamma}_y(0)=y$. The cut value $i_y$ of $y$ is defined by
\[
i_y:=\sup\{r: \text{ the segment }\gamma_y|_{[0,r]} \text{ is globally minimizing}  \}.
\]
The injectivity radius at $p$ is defined as $\mathfrak{i}_p:=\inf_{y\in S_pM} i_y$, whereas the cut locus of $p$ is
\[
\text{Cut}_p:=\left\{\gamma_y(i_y):\,y\in S_pM \text{ with }i_y<\infty \right\}.
\]
It should be remarked that $\text{Cut}_p$ is closed and null Lebesgue measure.

As in \cite{ZY}, fixing $p\in M$, let $(r,y)$ be the polar coordinate system at $p$. Note that $r(x)=\rho_+(x)$. Given an arbitrary measure $d\mathfrak{m}$, write
\[
d\mathfrak{m}=:\hat{\sigma}_p(r,y)dr\wedge d\nu_p(y),
\]
where $d\nu_p(y)$ is the Riemannian volume measure induced by $F$ on $S_pM$ . If $\mathbf{K}\leq k$, for each $y\in S_pM$, we have
\[
\Delta r=\frac{\partial}{\partial r}\log \hat{\sigma}_p(r,y)\geq \frac{\partial}{\partial r}\log \left( e^{-\tau(\dot{\gamma}_y(r))}\mathfrak{s}^{n-1}_k(r)  \right),\ 0<r<i_y,\tag{2.3}\label{1.2}
\]
and
\[
\lim_{r\rightarrow 0^+}\frac{\hat{\sigma}_p(r,y)}{e^{-\tau(\dot{\gamma}_y(r))}\mathfrak{s}^{n-1}_k(r)}=1,\tag{2.4}\label{new2.3}
\]
where $\gamma_y(t)$ is a geodesic with $\dot{\gamma}_y(0)=y$, and $\mathfrak{s}_k(t)$ is the unique
solution to $f'' + kf = 0$ with $f(0) = 0$ and $f'(0) = 1$.

Now we introduce the co-area formula in the Finsler setting. See \cite[Section 3.3]{Sh1} for more details.
Let $\varphi(x)$ be a  piecewise
$C^1$ function on $M$ such that every $\varphi^{-1}(t)$ is compact. The (area) measure on $\varphi^{-1}(t)$ is defined by
$dA:=(\nabla \varphi)\rfloor d\mathfrak{m}$. Then for any continuous function $f$ on $M$, we have the following co-area formula
\[
\int_Mf\, F(\nabla\varphi)\,d\mathfrak{m}=\int^{\infty}_{-\infty}\left( \int_{\varphi^{-1}(t)}f\, dA\right)dt.\tag{2.5}\label{1.3}
\]
Define the divergence of a vector field $X$ by
\[
\di(X)\, d\mathfrak{m}:=d\left( X\rfloor d\mathfrak{m}\right).
\]
Supposing $M$ is compact and oriented, we have the  divergence theorem
\[
\int_M\di(X)d\mathfrak{m}=\int_{\partial M} g_{\mathbf{n}}(\mathbf{n},X)\,d A,\tag{2.6}\label{1.4}
\]
where $dA=\mathbf{n}\rfloor d\mathfrak{m}$, and $\mathbf{n}$ is the unit outward normal vector field along $\partial M$, i.e., $F(\mathbf{n})=1$ and $ g_{\mathbf{n}}(\mathbf{n},Y)=0$ for any $Y\in T(\partial M)$.

Given a $C^2$-function $f$, set $\mathcal {U}=\{x\in M:\, df|_x\neq0\}$. The Laplacian of $f\in C^2(M)$ is defined on $\mathcal {U}$ by
\begin{align*}
\Delta f:=\text{div}(\nabla f)=\frac{1}{\sigma(x)}\frac{\partial}{\partial x^i}\left(\sigma(x)g^{*ij}(df|_x)\frac{\partial f}{\partial x^j}\right),\tag{2.7}\label{2.6'''}
\end{align*}
where $(g^{*ij})$ is the fundamental tensor of $F^*$. As in \cite{Ot}, we define
the distributional Laplacian of $u\in W^{1,2}_{\text{loc}}(M)$
in the weak sense by
\[
\int_M v{\Delta} u d\mathfrak{m}=-\int_M\langle \nabla u,dv\rangle d\mathfrak{m}, \text{ for all }v\in C^\infty_0(M).\tag{2.8}\label{2.7''}
\]
From (\ref{1.4}), one can see that (\ref{2.7''}) still makes sense when $M$ is a closed manifold (i.e., a compact manifold without boundary).

\section{Hardy inequalities on Finsler manifolds} \label{ns}
In this section, we investigate  Hardy inequalities on general Finsler manifolds and prove Theorem \ref{Th11}. Refer to \cite{FKV,KO,KR,YSK} for the Riemannian and specially Finslerian case.

Given a Finsler metric $F$, the {\it reverse} of $F$ is defined by $\widetilde{F}(x,y):=F(x,-y)$. It is not hard to see that $\widetilde{F}$ is also a Finsler metric. In this paper, we always use $\widetilde{*}$ to denote the quantity $*$ defined by $\widetilde{F}$.

Firstly, we investigate the relationship between $F$ and $\widetilde{F}$, which is important to establish the sharp Hard and Rellich inequalities in the non-reversible Finsler setting. Also see \cite{BCS,Ot} for more details.
\begin{lemma}\label{keylemma} Let $(M,F,d\mathfrak{m})$ be a Finsler manifold, where $d\mathfrak{m}$ is a measure on $M$. Then we have:

\noindent (1) $G^i(y)=\widetilde{G}^i(-y)$, $\mathbf{K}(P,y)=\widetilde{\mathbf{K}}(P,-y)$ and $-\mathbf{S}(y)=\widetilde{\mathbf{S}}(-y)$.

\noindent (2) $d_{F}(p,q)=d_{\widetilde{F}}(q,p)$ and hence, $(M,F)$ is forward (resp. backward) complete if and only if $(M,\widetilde{F})$ is backward (resp. forward)  complete.

\noindent (3) $\nabla (-f)=-\widetilde{\nabla} f$ and $\Delta (-f)=-\widetilde{\Delta}f$ for $f\in C^2(M)$.
\end{lemma}
\begin{proof}(i) It follows from the definition of $\widetilde{F}$ that $g_{ij}(y)=\widetilde{g}_{ij}(-y)$ and hence, $\tau(y)=\widetilde{\tau}(-y)$, $G^i(y)=\widetilde{G}^i(-y)$ and $\widetilde{R}_y=R_{-y}$, which implies $\mathbf{K}(P,y)=\widetilde{\mathbf{K}}(P,-y)$. In a local coordinate system, we can write
\[
\mathbf{S}(y)=\left\langle y^i\frac{\delta}{\delta x^i},d\tau\right\rangle,
\]
where $\frac{\delta}{\delta x^i}=\frac{\partial}{\partial x^i}-N^k_i \frac{\partial}{\partial y^k}$ and $N^i_j=\frac{\partial G^i}{\partial y^j}$. Since $\widetilde{\tau}(y)=\tau(-y)$, a direct calculation yields $\widetilde{\mathbf{S}}(y)=-\mathbf{S}(-y)$.

\noindent (ii) Fixing $p\in M$, recall that $\widetilde{\rho}_+(x)=d_{\widetilde{F}}(p,x)$ and $\rho_-(x)=d_F(x,p)$.
We claim that $\widetilde{\rho}_+(x)=\rho_-(x)$.

In fact, for any $x\in M$, there exists a sequence of  Lipshitz continuous curves $\tilde{\gamma}_n(t)$ of $\widetilde{F}$, $0\leq t\leq 1$ from $p$ to $x$ such that $L_{\widetilde{F}}(\widetilde{\gamma}_n)\rightarrow \widetilde{\rho}_+(x)$. Set $\gamma_n(t):=\widetilde{\gamma}_n(1-t)$. Then $\gamma_n$ is a Lipshitz continuous curve from $x$ to $p$ and
\begin{align*}
&\widetilde{\rho}_+(x)=\lim_{n\rightarrow \infty}\int^1_0\widetilde{F}(\dot{\widetilde{\gamma}}_n(t))dt=\lim_{n\rightarrow \infty}\int^1_0\widetilde{F}(-\dot{{\gamma}}_n(1-t))dt\\
=&\lim_{n\rightarrow \infty}\int^1_0{F}(\dot{{\gamma}}_n(1-t))dt=\lim_{n\rightarrow \infty}\int^1_0{F}(\dot{{\gamma}}_n(t))dt\geq\rho_-(x).
\end{align*}
The same method yields that $\rho_-(x)\geq \widetilde{\rho}_+(x)$ and hence, the claim is true, which is exactly $d_F(p,q)=d_{\widetilde{F}}(q,p)$. Therefore, a forward (resp., backward) convergent Cauchy sequence in $(M,F)$ is a backward (resp., forward) convergent Cauchy sequence in $(M,\widetilde{F})$, which implies (2).

\noindent (iii) Given $\eta\in T^*M$,
\[
\widetilde{F}^*(\eta)=\underset{X\in T_xM\backslash \{0\}}{\sup}\frac{\eta(X)}{\widetilde{F}(X)}=\underset{X\in T_xM\backslash \{0\}}{\sup}\frac{-\eta(-X)}{F(-X)}=F(-\eta),
\]
which implies that $g^{*ij}(\eta)=\widetilde{g}^{*ij}(-\eta)$. Thus, $\nabla (-f)=-\widetilde{\nabla} f$ follows from
\[
\nabla f=g^{*ij}(df)\frac{\partial f}{\partial x^i}\frac{\partial}{\partial x^j}.
\]
Now (\ref{2.6'''}) together with (\ref{2.7''}) and $\nabla (-f)=-\widetilde{\nabla} f$ yields $\Delta (-f)=-\widetilde{\Delta}f$.
\end{proof}

\begin{remark}\label{firstremark}An interesting phenomena is that  the
Busemann-Hausdorff measures (resp., the Holmes-Thompson measures) of $F$ and $\widetilde{F}$ coincide. For instance, given a Randers metric $F=\alpha+\beta$, it is not hard to see that its reverse metric $\widetilde{F}=\alpha-\beta$. Then it follows from (\ref{1.3**}) that $d\mathfrak{m}_{BH}=d\widetilde{\mathfrak{m}}_{BH}$ and $d\mathfrak{m}_{HT}=d\widetilde{\mathfrak{m}}_{HT}$.
\end{remark}

Given $u\in C(M)$, set
\begin{align*}
\rho_{u}(x):=\left\{
\begin{array}{lll}
& \rho_-(x), & \ \ \ \text{if } u(x)>0, \\
\\
& \rho_+(x), & \ \ \ \text{if } u(x)<0,\\
\\
&\frac12\left[\rho_+(x)+\rho_-(x)\right], & \ \ \ \text{if } u(x)=0.
\end{array}
\right.
\end{align*}
It should be remarked that if $F$ is reversible, $\rho_{u}(x)$ is exactly $\rho(x)$ for any $u\in C(M)$.
Furthermore, given $k,h\in \mathbb{R}$, we define a function $D_{k,h}(t)$ by
\[
D_{k,h}(t):=t\left(\frac{\mathfrak{s}'_k(t)}{\mathfrak{s}_k(t)}-h\right)-1, \text{ for }0< t<\frac{\pi}{2\sqrt{k}}.
\]
Here, $\frac{\pi}{2\sqrt{k}}:=\infty$ if $k\leq 0$.
It is not hard to check that $D_{k,h}(t)$ is non-negative if and only if $k\leq 0$, $h\leq 0$.
In this paper, an open domain is  a connected open set contained in some $n$-dimensional complete  Finsler manifold.
Now we have the following Hardy inequality.

\begin{theorem}\label{Hardy1}
Let $(M,F,d\mathfrak{m})$ be an $n$-dimensional complete Finsler manifold or an open domain containing $p$ with $\mathbf{K}\leq k$ and $|\mathbf{S}|\leq (n-1)h$. For any $\beta\in \mathbb{R}$ with $n-2>\beta$ and any $u\in C^\infty_0(M)$, we have
\begin{align*}
\int_M  \frac{F^{2}(\nabla u(x))}{\rho^\beta_u(x)}d\mathfrak{m}(x)&\geq\frac{(n-2-\beta)^2}{4}\int_M \frac{u^2(x)}{{\rho^{2+\beta}_u}(x)}d\mathfrak{m}(x)\\
&+\frac{(n-1)(n-2-\beta)}{2}\int_M  \frac{u^2(x)}{{\rho^{2+\beta}_u}(x)}\cdot D_{k,h}(\rho_u(x))d\mathfrak{m}(x),
\end{align*}
where the constant ${(n-2-\beta)^2}/{4}$ is sharp if $k\leq 0$ and $h= 0$.
\end{theorem}
\begin{proof}
\noindent \textbf{Step 1.}
Now set $\Omega_1:=\{x\in M:\, u(x)>0\}$, $\Omega_2:=\{x\in M:\, u(x)<0\}$, $\Omega_3:=\{x\in M:\, u(x)=0\}$. Note that $\Omega_i$, $1\leq i\leq 3$ are measurable.
Now we claim that for each $i$,
\begin{align*}
\int_{\Omega_i}  \frac{F^{2}(\nabla u(x))}{\rho^\beta_u(x)}d\mathfrak{m}(x)&\geq\frac{(n-2-\beta)^2}{4} \int_{\Omega_i} \frac{u^2(x)}{{\rho^{2+\beta}_u}(x)}d\mathfrak{m}(x)\\
&+\frac{(n-1)(n-2-\beta)}{2}\int_{\Omega_i}  \frac{u^2(x)}{{\rho^{2+\beta}_u}(x)}\cdot D_{k,h}(\rho_u(x))\,d\mathfrak{m}(x).\tag{3.1}\label{3.00}
\end{align*}

It is not hard to see that $(\ref{3.00})$ is true when $i=3$. We now show the case of $i=1$. The Cauchy inequality \cite[(1.2.3)]{BCS} implies (also see Theorem \ref{ineq})
\begin{align*}
F^{*2}(\eta)\geq F^{*2}(\xi)+2g^*_\xi(\xi,\eta-\xi),\,\forall \,\xi\neq0,\, \eta\in T^*M\tag{3.2}\label{3.1111}
\end{align*}
Set $v(x):={\rho_-}^\gamma(x)\cdot u(x)$, where $\gamma=(n-2-\beta)/2$. Thus, one gets
\[
du=v\cdot\gamma {\rho_-}^{-\gamma-1}\cdot (-d{\rho_-}) +{\rho_-}^{-\gamma}\cdot dv.
\]
Set $\eta:=du$ and $\xi:=v\cdot\gamma {\rho_-}^{-\gamma-1}\cdot (-d{\rho_-})$. Since $u,v$ are positive on $\Omega_1$, (\ref{3.1111}) together with  (\ref{1.1}) furnishes that
\begin{align*}
F^2(\nabla u)=F^{*2}(du)\geq\gamma^2\cdot v^2\cdot{\rho_-}^{-2\gamma-2}+2\gamma\cdot{\rho_-}^{-2\gamma-1}(x)\cdot v\cdot\langle \nabla ({-\rho_-}), dv\rangle
\end{align*}
holds almost everywhere on $\Omega_1$, which yields
\begin{align*}
\int_{\Omega_1}  \frac{F^{2}(\nabla u(x))}{\rho^\beta_-(x)}d\mathfrak{m}(x)\geq \gamma^2\int_{\Omega_1}  {\rho_-}^{-2\gamma-2-\beta}(x)\cdot v^2(x) d\mathfrak{m}(x)+R_0,
\end{align*}
where
\begin{align*}
R_0:=\gamma \int_{\Omega_1}  {\rho_-}^{-2\gamma-1-\beta}(x)\cdot \langle \nabla ({-\rho_-}), dv^2\rangle\, d\mathfrak{m}(x).\tag{3.3}\label{(2.1)}
\end{align*}

As before, let $\widetilde{F}$ denote the reverse of $F$ and set $\widetilde{\rho}_+(x)=d_{\widetilde{F}}(p,x)$.
Note that the  Legendre transformation  is positively homogeneous, i.e., for any $\lambda>0$, $\mathcal {L}^{-1}(\lambda \eta)=\lambda \mathcal {L}^{-1}( \eta)$, $\forall\,\eta\in T^*M$. Thus, for any $N>0$, we have
\begin{align*}
\nabla\rho^{-N}_-(x)&=\mathcal {L}^{-1}\left[ d(\rho^{-N}_-) \right]=\mathcal {L}^{-1}\left[ N \rho^{-N-1}_-d(-\rho_-) \right]\\
&=N \rho^{-N-1}_-\mathcal {L}^{-1}[d(-\rho_-)]=N \rho^{-N-1}_-\nabla(-\rho_-),\\
\Delta\rho^{-N}_-(x)&=\di(\nabla{\rho}^{-N}_-(x))=-\di\left( N \widetilde{\rho}^{-N-1}_+ \widetilde{\nabla}\widetilde{\rho}_+ \right)\\
&=-N\,\widetilde{\rho}^{-N-2}_+(x)\left[(-N-1)+ \widetilde{\rho}_+ \widetilde{\Delta} \widetilde{\rho}_+\right].\tag{3.4}\label{2.4}
\end{align*}
Then Lemma \ref{keylemma} together with (\ref{(2.1)}), (\ref{2.4})  and (\ref{2.7''}) furnishes
\begin{align*}
R_0&=\frac{\gamma}{2\gamma+\beta} \int_{\Omega_1}   \langle {\nabla} {{\rho}^{-2\gamma-\beta}_-}, dv^2\rangle\, d\mathfrak{m}(x)\\
&=\gamma\int_{\Omega_1}  v^2\cdot {\widetilde{\rho}_+}^{-2\gamma-2-\beta}(x)\left[-(n-1)+{\widetilde{\rho}_+}(x)\cdot\widetilde{\Delta} {\widetilde{\rho}_+}(x)\right]   d\mathfrak{m}(x).\tag{3.5}\label{2.3}
\end{align*}

It follows from Lemma \ref{keylemma} that  $\widetilde{\mathbf{K}}\leq k$  and $|\widetilde{\mathbf{S}}|\leq (n-1)h$. Hence, (\ref{1.2}) implies
\[
\widetilde{\Delta}\widetilde{\rho}_+(x)\geq (n-1)\left[ \frac{\mathfrak{s}'_k(\widetilde{\rho}_+(x))}{\mathfrak{s}_k(\widetilde{\rho}_+(x))}-h  \right] \text{ a.e. on }M,\tag{3.6}\label{2.5}
\]
which together with (\ref{2.3}) and (\ref{2.4}) furnishes
\[
R_0\geq(n-1)\gamma\int_{\Omega_1}  \frac{v^2(x)}{\widetilde{\rho}_+^{2\gamma+2+\beta}(x)}\cdot D_{k,h}\left(\widetilde{\rho}_+(x)\right)d\mathfrak{m}(x).
\]
Thus, (\ref{3.00}) is true for $i=1$.

In the case of $i=2$,  set $v=u\cdot \rho_+^\gamma$. Thus, we have
\[
du=(-v)\cdot \gamma\cdot \rho^{-\gamma-1}_+d\rho_++\rho^{-\gamma}_+dv,
\]
which together with (\ref{3.1111}) and Lemma \ref{keylemma} yields
\begin{align*}
F^2(du)\geq \gamma^2\cdot v^2\cdot\rho^{-2\gamma-2}_++2\gamma \widetilde{\rho}^{-2\gamma-1}_-\cdot v\cdot\langle \widetilde{\nabla}(-\widetilde{\rho}_-),dv\rangle.
\end{align*}
The rest of proof is similar as before, so we omit it. Thus, the desired inequality follows from (\ref{3.00}).

\noindent \textbf{Step 2.}  By the properties of $D_{k,h}$, it is easy to see that
\[
\frac{(n-1)(n-2-\beta)}{2}\int_M  \frac{u^2(x)}{{\rho^{2+\beta}_u}(x)}\cdot D_{k,h}(\rho_u(x))\,d\mathfrak{m}(x)\geq 0 \tag{3.7}\label{3.7777}
\]
if  $k\leq 0$ and $h= 0$.
In the following, we prove that the constant $\gamma^2=\left(\frac{n-2-\beta}{2} \right)^2$ is sharp in this case. Clearly, it suffices to show that
\[
\inf_{u\in C^\infty_0(M)\backslash\{0\}}\frac{\int_M  \frac{F^{2}(\nabla u)}{\rho^\beta_u(x)}d\mathfrak{m}}{\int_M \frac{u^2(x)}{{\rho^{2+\beta}_u}(x)}d\mathfrak{m}}=\gamma^2.
\]

Given $0<\epsilon<r<R<\widetilde{\mathfrak{i}}_p$, choose a cut-off function $\psi\in C^\infty_0(M)$ with $\text{supp}(\psi)=B^-_{p}(R)$ and $\psi|_{B^-_p(r)}\equiv1$. Here, $\widetilde{\mathfrak{i}}_p$ is the cut value of $p$ in $(M,\widetilde{F})$. Set $u_\epsilon(x):=\left[\max \{\epsilon, \rho_-(x)\}\right]^{-\gamma}$. Since $u:=\psi\cdot u_\epsilon\geq 0$ , we have
\begin{align*}
I_1(\epsilon):=&\int_M \frac{F^2(\nabla u(x))}{\rho^\beta_u(x)}d\mathfrak{m}(x)\\
=&\int_{B^-_p(r)}\frac{F^{*2}(d u_\epsilon(x))}{\rho^\beta_-(x)}d\mathfrak{m}(x)+\int_{B^-_p(R)\backslash B^-_p(r)}\frac{F^2(\nabla (\psi \rho^{-\gamma}_-)(x))}{\rho^\beta_-(x)}d\mathfrak{m}(x)\\
=&\int_{B^-_p(r)\backslash B^-_p(\epsilon)}\frac{F^{*2}(\gamma \rho^{-\gamma-1} d (-\rho_-))}{\rho^\beta_-(x)}d\mathfrak{m}(x)+\int_{B^-_p(R)\backslash B^-_p(r)}\frac{F^2(\nabla (\psi \rho^{-\gamma}_-))}{\rho^\beta_-(x)}d\mathfrak{m}(x)\\
=&:\gamma^2 \mathfrak{I}_1+\mathfrak{I}_2,\tag{3.8}\label{2.6}
\end{align*}
where
\begin{align*}
\mathfrak{I}_1:=\int_{B^-_p(r)\backslash B^-_p(\epsilon)}{\rho^{-n}_-}(x)d\mathfrak{m}(x),\ \mathfrak{I}_2:=\int_{B^-_p(R)\backslash B^-_p(r)}\frac{F^2(\nabla (\psi \rho^{-\gamma}_-))}{\rho^\beta_-(x)}d\mathfrak{m}(x).
\end{align*}
Clearly, $\mathfrak{I}_2$ is independent of $\epsilon$ and finite.
On the other hand, we have
\begin{align*}
I_2(\epsilon):=&\int_M \frac{u^2(x)}{{\rho^{2+\beta}_-}(x)}d\mathfrak{m} (x)\geq \int_{B^-_p(r)\backslash B^-_p(\epsilon)} \frac{(\psi u_\epsilon)^2(x)}{{\rho^{2+\beta}_-}(x)}d\mathfrak{m}(x)\\
=&\int_{B^-_p(r)\backslash B^-_p(\epsilon)} \frac{\rho^{-2\gamma}_-(x)}{{\rho^{2+\beta}_-}(x)}d\mathfrak{m}(x)=\mathfrak{I}_1.\tag{3.9}\label{2.7}
\end{align*}

We now estimate $\mathfrak{I}_1$. Lemma \ref{keylemma} points out $\rho_-(x)=\widetilde{\rho}_+(x)$ and hence, $B^-_p(l)=\widetilde{B}^+_p(l)$, where $\widetilde{B}^+_p(l)$ is the forward geodesic ball with radius $l$ centered at $p$ in $(M,\widetilde{F})$. The co-area formula (\ref{1.3}) then yields
\begin{align*}
\mathfrak{I}_1&=\int_{B^-_p(r)\backslash B^-_p(\epsilon)}{\rho^{-n}_-}(x)d\mathfrak{m}(x)=\int_{\widetilde{B}^+_p(r)\backslash \widetilde{B}^+_p(\epsilon)}\widetilde{\rho}^{-n}_+(x)\,d\mathfrak{m}(x)\\
&=\int^r_\epsilon dt \int_{\widetilde{S}^+_p(t)}t^{-n} d\widetilde{A}=\int^r_\epsilon t^{-n}\cdot \widetilde{A}(\widetilde{S}^+_p(t))\, dt,\tag{3.10}\label{2.8}
\end{align*}
where $\widetilde{S}^+_p(t)$ is the geodesic sphere $\{x\in M: \,\widetilde{\rho}_+(x)=t\}$ and $d\widetilde{A}:=\widetilde{\nabla} \widetilde{\rho}_+\rfloor d\mathfrak{m}$.

Now let $(t,y)$ be the polar coordinates at $p$ in $(M,\widetilde{F})$ and write $d\mathfrak{m}=\hat{\sigma}_p(t,y)dt\wedge d\nu_p(y)$, where $d\nu_p$ is the Riemannian measure of the indicatrix $\widetilde{S_pM}:=\{y\in T_pM:\,\widetilde{F}(p,y)=1\}$. Using (\ref{1.2}) again, we obtain
\[
\hat{\sigma}_p(t,y)\geq  e^{-\widetilde{\tau}(y)} \mathfrak{s}^{n-1}_k(t), \text{ for }0<t<\tilde{i}_y,
\]
where $\widetilde{\tau}$ is the distortion of $d\mathfrak{m}$ and $\tilde{i}_y$ is the cut value of $y$ in $(M,\widetilde{F})$.
Recall that $d\widetilde{A}=\widetilde{\nabla} \widetilde{\rho}_+\rfloor d\mathfrak{m}=\hat{\sigma}_p(t,y)d\nu_p(y)$. Thus, from above, we obtain that
\begin{align*}
\widetilde{A}(\widetilde{S}^+_p(t))&=\int_{\widetilde{S}^+_p(t)} d\widetilde{A}=\int_{\widetilde{S_pM}}\hat{\sigma}_p(t,y)d\nu_p(y)\geq \int_{\widetilde{S_pM}}e^{-\widetilde{\tau}(y)} \mathfrak{s}^{n-1}_k(t)d\nu_p(y)\\
&=C_p(d\mathfrak{m})\cdot\mathfrak{s}^{n-1}_k(t),\tag{3.11}\label{2.9}
\end{align*}
where
\[
C_p(d\mathfrak{m}):=\int_{\widetilde{S_pM}}e^{-\widetilde{\tau}(y)}d\nu_p(y).
\]
Clearly, $C_p(d\mathfrak{m})$ is a finite positive number since $\widetilde{S_pM}$ is compact.
In particular, Remark \ref{firstremark} implies $C_p(d\mathfrak{m})=n\omega_n$ (i.e., the area of $\mathbb{S}^{n-1}$) if $d\mathfrak{m}$ is the Busemann-Hausdorff measure (induced by $F$). Now (\ref{2.8}) combining with (\ref{2.9}) yields that
\begin{align*}
\mathfrak{I}_1\geq C_p(d\mathfrak{m})\int^r_\epsilon \frac1t\left[\frac{\mathfrak{s}_k(t)}{ t} \right]^{n-1}dt.\tag{3.12}\label{2.10}
\end{align*}
Obviously,
\[
\lim_{t\rightarrow 0^+}\left[\frac{\mathfrak{s}_k(t)}{ t} \right]^{n-1}=1,
\]
and hence, there exits a finite positive number $C(n,k,r)$ such that
\[
\left[\frac{\mathfrak{s}_k(t)}{ t} \right]^{n-1}\geq C(n,k,r), \text{ for $0<t\leq r$},
\]
which together with (\ref{2.10}) furnishes that
\[
\mathfrak{I}_1\geq C_p(d\mathfrak{m})\cdot C(n,k,r)\cdot\left[\ln r-\ln \epsilon \right]\rightarrow +\infty,\text{ as }\epsilon\rightarrow 0^+.\tag{3.13}\label{2.11}
\]
By (\ref{3.7777}), (\ref{2.6}), (\ref{2.7}), (\ref{2.10}) and (\ref{2.11}), we have
\begin{align*}
\gamma^2\leq \inf_{u\in C^\infty_0(M)\backslash\{0\}}\frac{\int_M  \frac{F^{2}(\nabla u)}{\rho^\beta_-(x)}d\mathfrak{m}}{\int_M \frac{u^2(x)}{{\rho^{2+\beta}_-}(x)}d\mathfrak{m}} \leq \lim_{\epsilon\rightarrow 0^+}\frac{I_1(\epsilon)}{I_2(\epsilon)}
=\lim_{\epsilon\rightarrow 0^+}\frac{\gamma^2\cdot\mathfrak{I}_1+\mathfrak{I}_2}{\mathfrak{I}_1}\leq \gamma^2.
\end{align*}
\end{proof}

\begin{remark}\label{reversibles-curvature}
One may notice that if $F$ is reversible, the Hardy inequality in Theorem \ref{Hardy1} is  sharp under the weaker assumption that $\mathbf{K}\leq k\leq 0$ and $\mathbf{S}\leq (n-1)h\leq 0$. However, due to Lemma \ref{keylemma},
if the  S-curvature of  a reversible Finsler manifold is non-positive (or nonnegative), then the S-curvature vanishes  factually.
\end{remark}

\begin{remark}
Given $0<\epsilon<r<R<\mathfrak{i}_p$, choose a cut-off function $\psi\in C^\infty_0(M)$ with $\text{supp}(\psi)=B^+_{p}(R)$ and $\psi|_{B^+_p(r)}\equiv1$. Here, ${\mathfrak{i}}_p$ is the cut value of $p$ in $(M,{F})$. Set $v_\epsilon(x):=-\left[\max \{\epsilon, \rho_+(x)\}\right]^{-\gamma}$ and $v:=\psi\cdot v_\epsilon$. One can use this function $v$ to show that the constant ${(n-2-\beta)^2}/{4}$ is sharp in Step 2 above.
\end{remark}

A Finsler-Hardmard $(M,F)$ is a simply connected forward complete Finsler manifold with $\mathbf{K}\leq 0$. The Cartan-Hadamard theorem in the Finsler setting (cf. \cite[Theorem 9.4.1]{BCS}) implies that $\mathfrak{i}_p=\infty$ for any $p\in M$ and hence, $M$ is noncompact. Thus, Theorem \ref{Hardy1}
furnishes \cite[Lemma 3.1]{KR}. In particular, their result is sharp.

Using Theorem \ref{Hardy1}, we get the following uncertainty principle type
inequality on a Finsler manifold.
\begin{corollary}
Let $(M,F,d\mathfrak{m})$ be  an $n$-dimensional  complete Finsler manifold or an open domain containing $p$ with $\mathbf{K}\leq 0$ and $\mathbf{S}=  0$. For any $\beta\in \mathbb{R}$ with $n-2>\beta$ and any $u\in C^\infty_0(M)$, we have
\begin{align*}
\left( \int_M \rho^{2+\beta}_u(x)\cdot u^2(x) d\mathfrak{m}(x) \right)^\frac12 \left(\int_M  \frac{F^{2}(\nabla u(x))}{\rho^{\beta}_u(x)}d\mathfrak{m}(x) \right)^\frac12\geq \left(\frac{n-2-\beta}{2}\right) \int_M  u^2(x) d\mathfrak{m}(x).
\end{align*}
\end{corollary}
\begin{proof}From Theorem \ref{Hardy1}, we have
\begin{align*}
\int_M  \frac{F^{2}(\nabla u(x))}{\rho^\beta_u(x)}d\mathfrak{m}(x)&\geq\left(\frac{n-2-\beta}{2} \right)^2\int_M \frac{u^2(x)}{{\rho^{2+\beta}_u}(x)}d\mathfrak{m}(x),
\end{align*}
which together with the H\"older inequality yields that
\begin{align*}
&\left( \int_M \rho^{2+\beta}_u(x)\cdot u^2(x) d\mathfrak{m} (x)\right)^\frac12 \left(\int_M  \frac{F^{2}(\nabla u)}{\rho^\beta_u(x)}d\mathfrak{m} (x)\right)^\frac12\\
\geq& \left(\frac{n-2-\beta}{2}\right)\left( \int_M \rho^{2+\beta}_u(x)\cdot u^2(x) d\mathfrak{m}(x) \right)^\frac12 \left(\int_M \frac{u^2(x)}{{\rho^{2+\beta}_u}(x)}d\mathfrak{m}(x)\right)^\frac12\\
=&\left(\frac{n-2-\beta}{2}\right) \int_M  u^2(x) d\mathfrak{m}(x).
\end{align*}
\end{proof}

\begin{remark}\label{3remark}
By combing the methods used in Theorem \ref{Hardy1} and \cite{K}, one can establish a sharp Heisenberg-Pauli-Weyl inequality on general Finsler manifolds. Refer to \cite{FS,K,KO} for more details about uncertainty principles.
\end{remark}

In the following, we will show that the Hardy inequality on Finsler manifolds can be refined by adding remainder terms like the Brezis-V\'azquez improvement if the manifold is of strictly negative flag curvature. Refer to \cite{BV,YSK} for the Euclidean and Riemannian case. Before doing this, we need the following inequality.

\begin{theorem}\label{ineq}Let $(M,F)$ be a Finsler manifold with finite uniformity constant $\Lambda_F<\infty$. Then
for any $\xi,\eta\in T^*_xM$,
\[
F^{*2}(\xi+\eta)\geq F^{*2}(\xi)+2g^*_\xi(\xi,\eta)+\frac{1}{\Lambda_F}F^{*2}(\eta).\tag{3.14}\label{3.11new}
\]
Here, we set $g^*_\xi(\xi,\eta)=0$ if $\xi=0$.
\end{theorem}
\begin{proof}
\noindent\textbf{ Step 1.} First we show that
\[
\sup_{\xi,\eta,\zeta\in T^*M\backslash\{0\}}\frac{g^*_\xi(\zeta,\zeta)}{g^*_\eta(\zeta,\zeta)}=\Lambda_F.
\]

Fixing any point $x\in M$ and choosing any $\eta\in T^*_xM\backslash\{0\}$, set $y=\mathcal {L}^{-1}(\eta)$. Let $\{e_i\}$  be a $g_y$-orthonormal basis of $T_xM$ and let $\{\theta^i\}$ be its dual basis. Since $\Lambda_F<\infty$, we have
\[
\frac{1}{\Lambda_F} \|X\|^2\leq g_{ij}(Z)X^iX^j\leq {\Lambda_F} \|X\|^2,\ \forall \,Z(\neq0),X\in T_xM,
\]
where
\[
\|X\|^2=\sum_i (X^i)^2=g_{ij}(y)X^iX^j.
\]
Thus, each eigenvalue of $(g_{ij}(Z))$ (and hence, each eigenvalue of $(g^{ij}(Z))$) is in $[\Lambda_F^{-1},\Lambda_F]$, which together with \cite[Lemma 3.1.2, (3.7)]{Sh1} yields that for any $\zeta=\zeta_i\, \theta^i\in T^*_xM$,
\[
\frac{1}{\Lambda_F}g^*_\eta(\zeta,\zeta)=\frac{1}{\Lambda_F}\delta^{ij}\zeta_i\zeta_j\leq g^{*ij}(\mathcal {L}(Z))\zeta_i\zeta_j\leq {\Lambda_F}\cdot \delta^{ij}\zeta_i\zeta_j={\Lambda_F}\cdot g^*_\eta(\zeta,\zeta).
\]
Since $Z$ is arbitrary and the Legendre transformation $\mathcal {L}: T_xM\rightarrow T^*_xM$ is a homeomorphism, the above inequality implies
\[
\Lambda_{F^*}:=\sup_{\xi,\eta,\zeta\in T^*M\backslash\{0\}}\frac{g^*_\xi(\zeta,\zeta)}{g^*_\eta(\zeta,\zeta)}\leq \Lambda_F<\infty.
\]
Using the same argument again, one can get $\Lambda_F\leq \Lambda_{F^*}$ and therefore, $\Lambda_{F^*}=\Lambda_F$.
In particular, for any $\eta\neq0$,
\[
g^*_\eta(\zeta,\zeta)\geq \frac{ F^{*2}(\zeta)}{\Lambda_F}.
\]

\noindent\textbf{Step 2.} Now we prove (\ref{3.11new}).
Without losing generality, we assume that $\eta\neq0$. Set $\zeta:=\eta/F(\eta)$.

\noindent(1) Suppose that $\xi+t\zeta\neq0$, for $t\in (0,F(\eta))$. Then we define a $C^2$-function
\[
f(t):=F^{*2}(\xi+t\zeta)-F^{*2}(\xi)-2g^*_\xi(\xi,t\zeta), \ t\in (0,F(\eta)).
\]
Clearly, $f(0)=0$ and
\begin{align*}
&f'(t)=2g^*_{\xi+t\zeta}(\xi+t\zeta,\zeta)-2g^*_\xi(\xi,\zeta), &f'(0)=0,\\
&f''(t)=2g^*_{\xi+t\zeta}(\zeta,\zeta)\geq \frac{2}{\Lambda_F}.
\end{align*}
Hence,
\begin{align*}
f'(t)=\int^t_0 f''(t)dt\geq \frac{2}{\Lambda_F}t, \
f(F^*(\eta))=\int^{F^*(\eta)}_0 f'(t)dt\geq \frac{F^{*2}(\eta)}{\Lambda_F},\tag{3.15}\label{new3.15}
\end{align*}
which is exactly (\ref{3.11new}).

\noindent(2) Suppose that there exits $t_0\in (0,F(\eta))$ such that $\xi+t_0\zeta=0$. Consider the Euclidean space $(T^*_xM,\hat{g}^*_x)$, where $\hat{g}^*$ is the average Riemannian metric induced by $F^*$. Fixing a  sufficiently small positive number $\epsilon(<1)$, let $C(t):t_0-\epsilon/\|\zeta\|\leq t \leq t_0+\epsilon/\|\zeta\|$ be a half-circle of radius $\epsilon$ centered at $0$ in $(T^*_xM,\hat{g}^*_x)$. In particular,
\[
C\left(t_0-\frac{\epsilon}{\|\zeta\|}\right)=\xi+\left(t_0-\frac{\epsilon}{\|\zeta\|}\right)\zeta,\ C\left(t_0+\frac{\epsilon}{\|\zeta\|}\right)=\xi+\left(t_0+\frac{\epsilon}{\|\zeta\|}\right)\zeta,
\]
where $\|\cdot\|$ is the norm induced by $\hat{g}^*_x$.
Thus,
\[
\|C(t)\|=\epsilon,\ \|\dot{C}(t)\|=\frac{\pi}{2}\|\zeta\|,\ \|\ddot{C}(t)\|=\frac{\pi^2}{4\epsilon}\|\zeta\|^2.
\]
Now set
\begin{align*}
\gamma(t):=\left\{
\begin{array}{lll}
\xi+t\zeta, &t\in [0,t_0-\epsilon/\|\zeta\|)\cup (t_0+\epsilon/\|\zeta\|,F(\eta)],\\
\\
C(t),& t\in \left[t_0-\epsilon/\|\zeta\|,t_0+\epsilon/\|\zeta\|\right],
\end{array}
\right.
\end{align*}
and
\[
f(t):=F^{*2}(\gamma(t))-F^{*2}(\xi)-2g^*_\xi(\xi,\gamma(t)-\xi).
\]
Hence, $f(0)=0$
and
\begin{align*}
f'(t)&=2g^*_{\gamma(t)}(\gamma(t),\dot{\gamma}(t))-2g^*_\xi(\xi,\dot{\gamma}(t)),\\
f''(t)&=2g^*_{\gamma(t)}(\dot{\gamma}(t),\dot{\gamma}(t))+2g^*_{\gamma(t)}(\gamma(t),\ddot{\gamma}(t))-2g_\xi(\xi,\ddot{\gamma}(t)).
\end{align*}
Then the Cauchy inequality \cite[(1.2.16)]{BCS} together with (\ref{2.000}) and Step 1 implies
\begin{align*}
\left|\left.f'\right|_{[t_0-\frac{\epsilon}{\|\zeta\|},t_0+\frac{\epsilon}{\|\zeta\|}]}\right|&\leq 2 |g^*_{C(t)}(C(t),\dot{C}(t))|+2|g^*_\xi(\xi,\dot{C}(t))|\\
&\leq 2 \Lambda^\frac32_F\cdot \|\dot{C}(t)\|\left(\|{C}(t)\|+F^*(\xi) \right)\leq \mathcal {C}_1(\Lambda_F, F^*(\xi),F^*(\zeta)),\\
\left|\left.f''\right|_{[t_0-\frac{\epsilon}{\|\zeta\|},t_0+\frac{\epsilon}{\|\zeta\|}]}\right|&\leq |2g^*_{C(t)}(\dot{C}(t),\dot{C}(t))|+|2g^*_{C(t)}(C(t),\ddot{C}(t))|+|2g_\xi(\xi,\ddot{C}(t))|\\
&\leq  2\Lambda^2_F\cdot \|\dot{C}(t)\|^2+2\Lambda^\frac32_F\cdot \|{C}(t)\| \cdot \|\ddot{C}(t)\|+2\Lambda_F\cdot F^*(\xi) \cdot \|\ddot{C}(t)\|\\
&\leq \epsilon^{-1}\cdot\mathcal {C}_2(\Lambda_F,F^*(\xi),F^*(\zeta)),
\end{align*}
where $\mathcal {C}_1(\Lambda_F, F^*(\xi),F^*(\zeta))$, $\mathcal {C}_2(\Lambda_F, F^*(\xi),F^*(\zeta))$ are two constants independent of $\epsilon$.
Therefore, for $t\in (t_0-\epsilon/\|\zeta\|,t_0+\epsilon/\|\zeta\|)$, we have
\begin{align*}
\left|\left.f'\right|_{[t_0-\frac{\epsilon}{\|\zeta\|},t_0+\frac{\epsilon}{\|\zeta\|}]}(t)\right|
&\leq \left|f'\left(t_0-\frac{\epsilon}{\|\zeta\|}\right)\right|+\left|\int_{t_0-\frac{\epsilon}{\|\zeta\|}}^t f''(s)ds\right|\\
&\leq \mathcal {C}_1(\Lambda_F, F^*(\xi),F^*(\zeta))+\frac{2{\Lambda^\frac12_F}\cdot\mathcal {C}_2(\Lambda_F, F^*(\xi),F^*(\zeta))}{F^*(\zeta)},
\end{align*}
which implies
\[
\lim_{\epsilon\rightarrow 0^+}\left|\int^{t_0+\frac{\epsilon}{\|\zeta\|}}_{t_0-\frac{\epsilon}{\|\zeta\|}} f'(t)dt\right|=0.\tag{3.16}\label{new3.16}
\]
It now follows from (\ref{new3.15}) and (\ref{new3.16}) that
\begin{align*}
f(F^*(\eta))&=\int^{F^*(\eta)}_{t_0+\frac{\epsilon}{\|\zeta\|}}+\int^{t_0+\frac{\epsilon}{\|\zeta\|}}_{t_0-\frac{\epsilon}{\|\zeta\|}}+\int^{t_0-\frac{\epsilon}{\|\zeta\|}}_0 f'(t)dt\\
&\geq\frac{F^{*2}(\eta)}{\Lambda_F}+\int^{t_0+\frac{\epsilon}{\|\zeta\|}}_{t_0-\frac{\epsilon}{\|\zeta\|}}f'(t)dt+\frac{(t_0-\frac{\epsilon}{\|\zeta\|})^2-(t_0+\frac{\epsilon}{\|\zeta\|})^2}{\Lambda_F}\rightarrow \frac{F^{*2}(\eta)}{\Lambda_F},
\end{align*}
as $\epsilon\rightarrow 0^+.$
\end{proof}

\begin{example}\label{firstex}
Let $F=\alpha+\beta$ be a Randers norm on $\mathbb{R}^n$. Then $(\mathbb{R}^n, F)$ is a locally Minkowski space.  Now we check that
\[
F^{*2}(\xi+\eta)\geq F^{*2}(\xi)+2g^*_\xi(\xi,\eta)+\frac{1}{\Lambda_F}F^{*2}(\eta),\ \forall\,\xi,\eta\in T^*\mathbb{R}^n.\tag{3.17}\label{888}
\]

For simplicity, set $b:=\|\beta\|_\alpha$. Thus, it follows from \cite[Corollary 5.2]{YZ} that
\[
\Lambda_F=\left(\frac{1+b}{1-b}\right)^2
\]
According to \cite[Example 3.1.1]{Sh1}, there exists a coordinate system $(x^i)$  on $\mathbb{R}^n$ with
\[
F^*(\xi)=|\xi|+b\cdot \xi_n,\ \forall\, \xi\in T^*\mathbb{R}^n,
\]
where $\xi=\xi_i dx^i$ and $|\xi|=\sqrt{\sum (\xi_i)^2}$.

\noindent \textbf{Case 1}. Clearly, (\ref{888}) holds when $\xi=0$. Here, we set $g_\xi(\xi,\cdot)=0$ if $\xi=0$.

\noindent \textbf{Case 2}. Suppose that $\xi\neq0$ and $\eta=-k\cdot \xi$, where $k\geq 1$. That is, there exists $t_0\in (0,1]$ such that $\xi+t_0 \eta=0$. Without loss of generality, we assume that
$F^*(-\xi)\leq F^*(\xi)$.
Thus,
a direct calculation  yields
\begin{align*}
&F^{*2}(\xi+\eta)-\left[ F^{*2}(\xi)+2g^*_\xi(\xi,\eta)+\frac{1}{\Lambda_F}F^{*2}(\eta)\right]\\
=&F^{*2}(\xi)\left[(2k-1)+\left(-k^2\left(\frac{1-b}{1+b}\right)^2+(k-1)^2 \right)\frac{F^{*2}(-\xi)}{F^{*2}(\xi)}  \right]\\
\geq& F^{*2}(\xi)\left[(2k-1)-k^2+(k-1)^2  \right]\geq 0.
\end{align*}

\noindent \textbf{Case 3}. Suppose that $\gamma(t)=\xi+t\eta$, $0\leq t\leq 1$, does not contain $0$. Thus,
set
\begin{align*}
f(t):=F^{*2}(\gamma(t))-\left[ F^{*2}(\xi)+2g^*_\xi(\xi,\gamma(t)-\xi)\right].
\end{align*}
A direct calculation together with \cite[p.283]{BCS} yields that
\begin{align*}
f'(t)&=2F^*(\xi+t\eta)\left[\frac{\langle\xi+t\eta,\eta\rangle}{|\xi+t\eta|}+b\eta_n \right]-2F^*(\xi)\left[\frac{\langle\xi,\eta\rangle}{|\xi|}+b\eta_n  \right],\\
f''(t)&=2\left\{\frac{F^*(\xi+t\eta)}{|\xi+t\eta|}\left[ |\eta|^2-\frac{\langle \xi+t\eta,\eta\rangle^2}{|\xi+t\eta|^2}\right]+\left[ \frac{\langle\xi+t\eta,\eta\rangle}{|\xi+t\eta|} +b\eta_n\right]^2\right\},
\end{align*}
where $\langle\xi,\eta\rangle:=\sum \xi_i\eta_i$.

Since $\xi+t\eta\neq0$, set $\zeta=(\xi+t\eta)/|\xi+t\eta|=:s\eta/|\eta|+\eta^\perp$, where $s\in [-1,1]$ and $\langle\eta^\perp,\eta\rangle=0$. Then we obtain that
\begin{align*}
f''(t)=&2\left\{[1+b \zeta_n]\left[|\eta|^2-\langle \zeta,\eta\rangle^2\right]+[\langle \zeta,\eta\rangle+b\eta_n]^2\right\}\\
=&2|\eta|^2\left[(1-s^2)(1+b\zeta_n)+(s+l)^2  \right],
\end{align*}
where $l=b\eta_n/|\eta|$. Since $s,\zeta_n\in [-1,1]$ and $l\in [-b,b]$, one gets
\begin{align*}
(1-s^2)(1+b\zeta_n)+(s+l)^2\geq (1-b)^2.
\end{align*}
Hence,
\begin{align*}
f''(t)&\geq2 |\eta|^2(1-b)^2=2(1-b)^2\frac{|\eta|^2}{F^{*2}(\eta)}F^{*2}(\eta)\\
&\geq2 \left(\frac{1-b}{1+b}\right)^2 F^{*2}(\eta)=\frac{2}{\Lambda_F}F^{*2}(\eta).\tag{3.18}\label{3.18new}
\end{align*}
Then one can obtain (\ref{888}) by integrating (\ref{3.18new}) twice.
\end{example}

\begin{lemma}\label{volume}Let $(M,F,d\mathfrak{m})$ be an $n$-dimensional forward complete Finsler manifold or an open domain containing $p$ with $\mathbf{K}\leq k<0$,  $\mathbf{S}=0$ and finite reversibility $\lambda_F<\infty$.
Thus,  there exists a positive constant $C=C(n,\lambda_F,k)$ such that for any $u\in C^\infty_0(M)$,
\[
 \int_M\frac{v^2(x)}{\rho^{n-2}_u(x)}d\mathfrak{m}(x)\leq C \int_M \frac{F^2(\nabla v)}{\rho^{n-2}_u(x)}d\mathfrak{m}(x),
\]
where $v:=u \cdot \rho_u^\gamma$ and $\gamma$ is a  constant.
\end{lemma}

\begin{proof}The proof is inspired by \cite[Lemma
4.3]{YSK}. Set $\Omega_1:=\{x\in M:\, v(x)>0\}$, $\Omega_2:=\{x\in M:\, v(x)<0\}$ and $\Omega_3:=\{x\in M:\, v(x)=0\}$. Obviously,
it suffices to show that
\[
\int_{\Omega_i}\frac{v^2(x)}{\rho^{n-2}_v(x)}d\mathfrak{m}(x)\leq \left(\frac{2\lambda_F}{\sqrt{|k|}\cdot\min\{1,n-1 \}} \right)^2\int_{\Omega_i}\frac{F^2(\nabla v)}{\rho^{n-2}_v(x)}d\mathfrak{m}(x),\tag{3.19}\label{3.19new}
\]
for $1\leq i\leq 3$.

It is easy to see that (\ref{3.19new}) is true when $i=3$. Now we show the case of $i=1$.
Since $\Omega_1$ is open, we can write
\[
\Omega_1=\left\{ \widetilde{\exp}_p(sy):\ y\in\widetilde{\mathcal {S}}\subset \widetilde{S_pM},\ t^1_y< s< t^2_y,\ldots , t^{l-1}_y<s<t^l_y,\right\},
\]
where $t^\alpha_y\leq \tilde{i}_y$, $\alpha=1,2,\ldots,l$ and $\widetilde{\exp}_p$ is the exponential map at $p$ in $(M,\widetilde{F})$. In particular,
\[
\lim_{s\rightarrow t^\alpha_y}v\circ\widetilde{\exp}_p(s y)=0.\tag{3.20}\label{3.11'}
\]

Let $(r,y)$ denote the polar coordinate system of $(M,\widetilde{F})$ at $p$.
Recall that $\rho_-(x)=r(x)$. Set $d\mathfrak{m}:=\hat{\sigma}_p(r,y)dr\wedge d\nu_p(y)$, $J(r,y):=r^{1-n}\hat{\sigma}_p(r,y)$ and $J_k(r,y):=r^{1-n}\mathfrak{s}^{n-1}_k(r)$. It follows from (\ref{1.2}) that
\[
\frac{\partial}{\partial r}\log \hat{\sigma}(r,y)\geq \frac{\partial}{\partial r}\log \mathfrak{s}_k^{n-1}(r),\ 0<r< \tilde{i}_y,\tag{3.21}\label{3.21}
\]
which implies that for $ 0<r<\tilde{i}_y$,
\begin{align*}
&\frac{\partial_r(rJ(r,y))}{rJ(r,y)}
\geq\frac{1}{r}+\frac{\partial_rJ_k(r,y)}{J_k(r,y)}\\
\geq &\min\{1,n-1\}\left[ \frac1r+\left(\sqrt{|k|}\coth (\sqrt{|k|} r)-\frac1r \right) \right]\\
\geq &\sqrt{|k|}\cdot\min\{1,n-1\}\tag{3.22}\label{3.22}.
\end{align*}
Thus, (\ref{3.22}) together with (\ref{3.11'}) yields
\begin{align*}
&\sqrt{|k|}\cdot\min\{1,n-1\}\int_{\widetilde{\mathcal {S}}}d\nu_p(y)\left(\sum_\alpha\int_{t^{\alpha}_y}^{t^{\alpha+1}_y}r\cdot J(r,y)\cdot v^2 dr\right)\\
\leq &\int_{\widetilde{\mathcal {S}}}d\nu_p(y)\left(\sum_\alpha\int_{t^{\alpha}_y}^{t^{\alpha+1}_y}\partial_r(rJ(r,y))\cdot v^2 dr\right)\\
=& -\int_{\widetilde{\mathcal {S}}}d\nu_p(y)\left(\sum_\alpha\int_{t^{\alpha}_y}^{t^{\alpha+1}_y} rJ(r,y)\cdot\partial_r(v^2) dr\right)\\
\leq &2\int_{\widetilde{\mathcal {S}}}d\nu_p(y)\left(\sum_\alpha\int_{t^{\alpha}_y}^{t^{\alpha+1}_y} rJ(r,y)\cdot|v||\partial_r v| dr\right).\tag{3.23}\label{2.12'}
\end{align*}
Note that
\begin{align*}
\partial_rv&=\langle\partial_r,dv\rangle\leq \widetilde{F}(\partial_r)\widetilde{F}^*(dv)=\widetilde{F}^*(dv)\leq \lambda_F\cdot F^*(dv)=\lambda_F\cdot {F}(\nabla v),\\
-\partial_rv&=\langle\partial_r,d(-v)\rangle\leq \widetilde{F}(\partial_r)\widetilde{F}^*(d(-v))={F}^*(dv)={F}(\nabla v).
\end{align*}
Hence, $|\partial_rv|\leq \lambda_F\cdot {F}(\nabla v)$. It follows from (\ref{2.12'}) and the H\"older inequality  that
\begin{align*}
&\sqrt{k}\cdot\min\{1,n-1\}\int_{\widetilde{\mathcal {S}}}d\nu_p(y)\left(\sum_\alpha\int_{t^{\alpha}_y}^{t^{\alpha+1}_y}r J(r,y)\cdot v^2 dr\right)\\
\leq& 2\,\lambda_F\int_{\widetilde{\mathcal {S}}}d\nu_p(y)\left(\sum_\alpha\int_{t^{\alpha}_y}^{t^{\alpha+1}_y} rJ(r,y)\cdot|v|\cdot F(\nabla v) \,dr\right)\\
\leq &2\,\lambda_F \left[ \int_{\widetilde{\mathcal {S}}}d\nu_p(y)\left(\sum_\alpha\int_{t^{\alpha}_y}^{t^{\alpha+1}_y} rJ(r,y)\cdot v^{2} \,dr\right) \right]^{\frac{1}2}\\
&\cdot \left[ \int_{\widetilde{\mathcal {S}}}d\nu_p(y)\left(\sum_\alpha\int_{t^{\alpha}_y}^{t^{\alpha+1}_y} rJ(r,y)\cdot F^{2}(\nabla v) \,dr\right) \right]^{\frac{1}2}.
\end{align*}
That is,
\[
\int_{\Omega_1}\frac{v^2(x)}{\rho^{n-2}_v(x)}d\mathfrak{m}(x)\leq \left(\frac{2\lambda_F}{\sqrt{|k|}\cdot\min\{1,n-1 \}} \right)^2\int_{\Omega_1}\frac{F^2(\nabla v)}{\rho^{n-2}_v(x)}d\mathfrak{m}(x) .
\]
Likewise, one can show that (\ref{3.19new}) is true on $\Omega_2$.
\end{proof}

\begin{remark}
Under the same assumption, one can show that given $q>1$,
 \[
\int_M\frac{u^q(x)}{\rho^{n-q}_\pm(x)}d\mathfrak{m}(x)\leq \left(\frac{q\lambda_F}{\sqrt{|k|}\cdot\min\{n-1,q-1 \}} \right)^q\int_M\frac{F^q(\nabla u)}{\rho^{n-q}_\pm(x)}d\mathfrak{m}(x),
\]
for any $u\in C^\infty_0(M)$. In particular,
set
\[
\mu_\pm(M):=\inf_{f\in C^\infty_0(M)\setminus\{0\}}\frac{\int_M\frac{F^2(\nabla f)}{\rho^{n-2}_\pm(x)}d\mathfrak{m}(x)}{\int_M\frac{f^2(x)}{\rho^{n-2}_\pm(x)}d\mathfrak{m}(x)}.
\]
Then $\mu_\pm(M)>0$.
 See \cite[Lemma
4.3]{YSK} for the Riemannian case.
\end{remark}

Now we have the following refined Hardy inequality.
\begin{theorem}\label{Hardy2}
Let $(M,F,d\mathfrak{m})$ be an $n$-dimensional  complete Finsler manifold or an open domain containing $p$ with $\mathbf{K}\leq k<0$,  $\mathbf{S}= 0$ and finite uniformity constant $\Lambda_F<\infty$. There exists a positive constant $C=C(n,\Lambda_F,k)$ such that for any $\beta\in \mathbb{R}$ with $n-2>\beta$ and any $u\in C^\infty_0(M)$,
\begin{align*}
\int_M  \frac{F^{2}(\nabla u)}{\rho^\beta_u(x)}d\mathfrak{m}&\geq\frac{(n-2-\beta)^2}{4} \int_M \frac{u^2(x)}{{\rho^{2+\beta}_u}(x)}d\mathfrak{m}(x)\\
&+\frac{(n-1)(n-2-\beta)}{2}\int_M  \frac{u^2(x)}{{\rho^{2+\beta}_u}(x)}\cdot D_{k,0}(\rho_u(x))d\mathfrak{m}(x)\\
&+\frac{C}{\Lambda_F}\int_M\frac{u^2(x)}{\rho^{\beta}_u(x)}d\mathfrak{m}(x),
\end{align*}
where the constant $\frac{(n-2-\beta)^2}{4} $ is sharp.
\end{theorem}
\begin{proof}Set $\Omega_1:=\{x\in M:\, u(x)>0\}$, $\Omega_2:=\{x\in M:\, u(x)<0\}$, $\Omega_3:=\{x\in M:\, u(x)=0\}$.

Now we consider the case of $\Omega_1$.
Set $v(x):=u(x)\cdot{\rho_-}^\gamma(x)$, where $\gamma=\frac{n-2-\beta}{2}$. Thus,
\[
du=v\cdot\gamma {\rho_-}^{-\gamma-1}\cdot (-d{\rho_-}) +{\rho_-}^{-\gamma}\cdot dv.
\]
Let $\xi:=v\cdot\gamma {\rho_-}^{-\gamma-1}\cdot (-d{\rho_-})$, $\eta:={\rho_-}^{-\gamma}\cdot dv$. Putting $\xi,\eta$ into (\ref{3.11new}) and using the method in Step 1 of Theorem \ref{Hardy1}, one gets
\begin{align*}
\int_{\Omega_1}  \frac{F^{2}(\nabla u)}{\rho^\beta_u(x)}d\mathfrak{m}(x)\geq\frac{(n-2-\beta)^2}{4} \int_{\Omega_1} \frac{u^2(x)}{{\rho^{2+\beta}_u}(x)}d\mathfrak{m}(x)+R_0+R_1,\tag{3.24}\label{3.24}
\end{align*}
where
\begin{align*}
R_0&\geq \frac{(n-1)(n-2-\beta)}{2}\int_{\Omega_1}  \frac{u^2(x)}{{\rho^{2+\beta}_u}(x)} D_{k,0}(\rho_u(x))d\mathfrak{m}(x),\\
R_1&= \frac{1}{\Lambda_F}\int_{\Omega_1} \frac{F^{*2}(dv)}{\rho^{2\gamma+\beta}_u(x)}d\mathfrak{m}(x)=\frac{1}{\Lambda_F}\int_{\Omega_1} \frac{F^{2}(\nabla v)}{\rho^{n-2}_u(x)}d\mathfrak{m}(x).
\end{align*}
Similarly, one can establish (\ref{3.24}) on $\Omega_2$ and $\Omega_3$. Hence,
\begin{align*}
\int_M  \frac{F^{2}(\nabla u)}{\rho^\beta_u(x)}d\mathfrak{m}(x)&\geq\frac{(n-2-\beta)^2}{4} \int_M \frac{u^2(x)}{{\rho^{2+\beta}_u}(x)}d\mathfrak{m}(x)\\
&+\frac{(n-1)(n-2-\beta)}{2}\int_M  \frac{u^2(x)}{{\rho^{2+\beta}_u}(x)}\cdot D_{k,0}(\rho_u(x))d\mathfrak{m}(x)\\
&+\frac{1}{\Lambda_F}\int_{M} \frac{F^{2}(\nabla v)}{\rho^{n-2}_u(x)}d\mathfrak{m}(x).
\end{align*}
Since $\lambda_F\leq \sqrt{\Lambda_F}$, it follows from Lemma \ref{volume} that
\begin{align*}
\int_{M} \frac{F^{2}(\nabla v)}{\rho^{n-2}_u(x)}d\mathfrak{m}(x)\geq {C(n,\Lambda_F,k)}\int_M\frac{u^2(x)}{\rho^{\beta}_u(x)}d\mathfrak{m}(x)
\end{align*}
and the desired inequality follows. The rest of the proof is the same as the one in Step 2 of Theorem \ref{Hardy1}.
\end{proof}

\begin{proof}[The proof of Theorem \ref{Th11}] It is easy to see that the first consequence of Theorem \ref{Th11} follows from Theorem \ref{Hardy1} while the second one is exactly Theorem \ref{Hardy2}.
\end{proof}

\section{Rellich inequalities on Finsler manifolds} \label{ns}
This section is devoted to sharp Rellich inequalities on non-reversible Finsler manifolds. Let $(M,F)$ be an $n$-dimensional  complete Finsler manifold or an open domain containing $p$.
For $k\geq 2$, set
\[
C^k_{0,F,d\mathfrak{m},\beta}(M):=\left\{u\in C^k_0(M):\,G^\beta_{F,d\mathfrak{m}}(u)=0\right\},
\]
where
\[
G^\beta_{F,d\mathfrak{m}}(u):=\int_M\left[ u^2(x)\cdot \varrho_{u,\beta}(x)+2\rho^{-\beta-2}_u(x)\cdot \di (u\nabla u)  \right]d\mathfrak{m}(x),
\]
and
\begin{align*}
\varrho_{u,\beta}(x):=\left\{
\begin{array}{lll}
& -\Delta(\rho^{-\beta-2}_-)(x), & \ \ \ \text{if } u(x)>0, \\
\\
& \Delta(-\rho^{-\beta-2}_+)(x), & \ \ \ \text{if } u(x)<0,\\
\\
&\frac12\left[-\Delta(\rho^{-\beta-2}_-)(x) +\Delta(-\rho^{-\beta-2}_+)(x)\right], & \ \ \ \text{if } u(x)=0.
\end{array}
\right.
\end{align*}
It is not hard to see that $\varrho_{u,\beta}=-\Delta (\rho^{-\beta-2})$ and $2\di (u\nabla u) =\Delta (u^2)$ if $F$ is reversible and hence, $C^k_0(M)=C^k_{0,F,d\mathfrak{m},\beta}(M)$ in the Riemannian setting.
It should be remarked that $2\di (u\nabla u) \neq\Delta (u^2)$ on a non-reversible Finsler manifold if $u<0$ and hence, $C^k_{0,F,d\mathfrak{m},\beta}(M)\varsubsetneq  C^k_0(M)$  in general. However, the following result implies that $C^k_{0,F,d\mathfrak{m},\beta}(M)$ is not empty.
\begin{proposition}\label{nonempty}
Suppose that $u(x)=f\circ\rho_-(x)\in C^k_0(M)$ is a non-negative function. If $\frac{df(t)}{dt}\leq 0$   holds almost everywhere, then $G^\beta_{F,d\mathfrak{m}}(u)=0$.
\end{proposition}
\begin{proof}Suppose $u=f\circ\rho_-(x)\geq 0$. Set $\Omega_1:=\{x\in M:\, u(x)>0\}$. Thus,
it follows from Lemma \ref{keylemma},  (\ref{2.7''}) and (\ref{2.4}) that
\begin{align*}
&G^\beta_{F,d\mathfrak{m}}(u)
=\int_{\Omega_1}\left[ -u^2(x)\cdot {\Delta}({\rho}^{-\beta-2}_-)(x)-\widetilde{\rho}^{-\beta-2}_+(x)\cdot\widetilde{\Delta}(-u^2(x))   \right]d\mathfrak{m}(x)\\
=&-(\beta+2) \int_{\Omega_1} \widetilde{\rho}^{-\beta-3}_+(x)\left[\langle \widetilde{\nabla}\widetilde{\rho}_+,d u^2(x) \rangle+\langle\widetilde{\nabla}(-u^2(x)) ,d\widetilde{\rho}_+    \rangle   \right]d\mathfrak{m}(x).\tag{4.1}\label{3.4'''}
\end{align*}
Let $(r,\theta^\alpha)$ be the polar coordinate system at $p$ on $(M,\widetilde{F})$.
Then $\frac{\partial}{\partial r}=\widetilde{\nabla}\widetilde{\rho}_+$ and
\begin{align*}
d(-u^2(x))&=-2f(\widetilde{\rho}_+(x))\,f'(\widetilde{\rho}_+(x))\, d\widetilde{\rho}_+=-2f(r)\,f'(r)\,dr,\\
\widetilde{\nabla}(-u^2(x))&=\widetilde{g}^{11}_{\widetilde{\nabla}(-u^2(x))}\frac{d (-f^2)}{d r}\frac{\partial}{\partial r}=\widetilde{g}^{*11}_{-2f(r)\,f'(r)\,dr}(-2f(r)f'(r))\frac{\partial}{\partial r}\\
&=-2\widetilde{g}^{*11}_{dr}\,f(r)\, f'(r)\frac{\partial}{\partial r}=-2f(r)\,f'(r)\frac{\partial}{\partial r},
\end{align*}
where
\[
\widetilde{g}^{*11}_{dr}=\widetilde{g}^*_{dr}(dr,dr)=\widetilde{F}^{*2}(dr)=1.
\]
Hence,
\begin{align*}
\langle \widetilde{\nabla}\widetilde{\rho}_+,d u^2(x) \rangle+\langle\widetilde{\nabla}(-u^2(x)) ,d\widetilde{\rho}_+    \rangle=2f(r)\,f'(r)-2f(r)\,f'(r)=0,
\end{align*}
which together with (\ref{3.4'''}) yields $G^\beta_{F,d\mathfrak{m}}(u)=0$.
\end{proof}

\begin{remark}
From the result above, we can construct infinitely many $u\in C^k_{0,F,d\mathfrak{m},\beta}(M)$. For example,  after fixing a positive number $R< \widetilde{\mathfrak{i}}_p$ with $\overline{B^-_p(R)} \subset M$, choose a function $f\in C^k_0([0,+\infty))$ such that $f|_{t\leq R/2}=1$, $f|_{t\geq R}=0$ and $f'(t)\leq 0$. Then $u(x):=f\circ \rho_-(x)\in  C^k_{0,F,d\mathfrak{m},\beta}(M)$.
\end{remark}

Now we have the following theorem. Refer to \cite{KO,KR,YSK} for the reversible case.
\begin{theorem}\label{Rll1}
Let $(M,F,d\mathfrak{m})$  be an $n$-dimensional  complete  Finsler manifold or an open domain containing $p$ with $\mathbf{K}\leq k$ and $|\mathbf{S}|\leq (n-1)h$.
For any $\beta\in \mathbb{R}$ with $-2<\beta<n-4$ and any $u\in C^\infty_{0,F,d\mathfrak{m},\beta}(M)$, we have
\begin{align*}
\int_M\frac{(\Delta u)^2}{\rho^\beta_u(x)} d\mathfrak{m}(x)&\geq \frac{(n-4-\beta)^2(n+\beta)^2}{16}\int_M \frac{u^2(x)}{{\rho^{4+\beta}_u}(x)}d\mathfrak{m}(x)\\
&+\frac{(n-4-\beta)(n+\beta)(n-1)(n-2)}{4}\int_M  \frac{u^2(x)}{{\rho^{4+\beta}_u}(x)}\cdot D_{k,h}(\rho_u(x))d\mathfrak{m}(x),
\end{align*}
where the constant $\frac{(n-4-\beta)^2(n+\beta)^2}{16}$ is sharp if $k\leq 0$ and $h= 0$.

\end{theorem}

\begin{proof}
\noindent \textbf{Step 1.} Set $\Omega_1:=\{x\in M:\, u(x)>0\}$, $\Omega_2:=\{x\in M: \,u(x)<0\}$ and $\Omega_3:=\{x\in M:\, u(x)=0\}$.  Obviously, one has
\begin{align*}
-\varrho_{u,\beta}(x)=\left\{
\begin{array}{lll}
& \Delta(\rho^{-\beta-2}_-)(x), & \ \ \ \text{if } x\in \Omega_1, \\
\\
& -\Delta(-\rho^{-\beta-2}_+)(x)=\widetilde{\Delta}(\widetilde{\rho}^{-\beta-2}_-)(x), & \ \ \ \text{if }  x\in \Omega_2.\\
\end{array}
\right.
\end{align*}
Set $\gamma=\frac{n-4-\beta}{2}>0$. Then Lemma \ref{keylemma} together with (\ref{2.5}), (\ref{2.4}) and (\ref{1.2}) yields
\[
-\varrho_{u,\beta}(x)\leq  (-2-\beta)\cdot{\rho}^{-4-\beta}_u(x)\cdot\left[ 2\gamma+(n-1)\cdot D_{k,h}(\rho_u(x)) \right], \ x\in \Omega_1\cup \Omega_2.\tag{4.2}\label{*}
\]
Hence, given $u\in C^\infty_{0,F,d\mathfrak{m},\beta}(M)$, we have
\begin{align*}
&\int_M u^2(x)\cdot\varrho_{u,\beta}(x)d\mathfrak{m}(x)\\
\geq &(2+\beta)\int_M\frac{u^2(x)}{{\rho}^{4+\beta}_u(x)}\left[ 2\gamma+(n-1)\cdot D_{k,h}(\rho_u(x)) \right]d\mathfrak{m}(x).\tag{4.3}\label{3.1}
\end{align*}
Clearly,
\begin{align*}
&\int_M 2\rho^{-\beta-2}_u(x)\cdot\di(u \nabla u)d\mathfrak{m}(x)\\
=&2\int_M \rho^{-\beta-2}_u(x)\cdot F^2(\nabla u)d\mathfrak{m}(x)+2\int_M\rho^{-\beta-2}_u(x)\cdot u\Delta u d\mathfrak{m}(x),
\end{align*}
which together with (\ref{3.1}) and $G^\beta_{F,d\mathfrak{m}}(u)=0$ yields that
\begin{align*}
-\int_M\frac{u\Delta u}{\rho^{\beta+2}_u(x)} d\mathfrak{m}(x)\geq& \left(\frac{2+\beta}{2}\right)\int_M\frac{u^2(x)}{{\rho}^{4+\beta}_u(x)}\cdot\left[ 2\gamma+(n-1)\cdot D_{k,h}(\rho_u(x)) \right]d\mathfrak{m}(x)\\
&+\int_M \frac{F^2(\nabla u)}{\rho^{\beta+2}_u(x)}d\mathfrak{m}(x).\tag{4.4}\label{3.2}
\end{align*}
Using Theorem \ref{Hardy1}, we have
\begin{align*}
\int_M  \frac{F^2(\nabla u)}{\rho^{\beta+2}_u(x)}d\mathfrak{m}(x)& \geq\gamma^2\int_M \frac{u^2(x)}{{\rho^{4+\beta}_u}(x)}d\mathfrak{m}(x)\\
&+(n-1)\gamma\int_M  \frac{u^2(x)}{{\rho^{4+\beta}_u}(x)}\cdot D_{k,h}(\rho_u(x))d\mathfrak{m}(x),
\end{align*}
which together with (\ref{3.2}) furnishes that
\begin{align*}
-\int_M\frac{u\Delta u}{\rho^{\beta+2}_u(x)} d\mathfrak{m}(x)\geq\frac{\gamma(n+\beta)}{2} A+\frac{(n-1)(n-2)}{2} B,\tag{4.5}\label{3.3}
\end{align*}
where
\begin{align*}
A:=\int_M \frac{u^2(x)}{{\rho^{4+\beta}_u}(x)}d\mathfrak{m}(x),\ B:=\int_M  \frac{u^2(x)}{{\rho^{4+\beta}_u}(x)}\cdot D_{k,h}(\rho_u(x))d\mathfrak{m}(x).
\end{align*}
By the H\"older inequality, we get
\begin{align*}
-\int_M\frac{u\Delta u}{\rho^{\beta+2}_u(x)} d\mathfrak{m}(x)\leq \int_M \frac{|u \Delta u|}{\rho^{\beta+2}_u(x)}d\mathfrak{m}(x)
\leq  \left( \int_M\frac{(\Delta u)^2}{\rho^\beta_u(x)} d\mathfrak{m}(x)\right)^\frac12\cdot A^\frac12,
\end{align*}
which together with (\ref{3.3}) furnishes
\begin{align*}
\int_M\frac{(\Delta u)^2}{\rho^\beta_u(x)} d\mathfrak{m}(x)\geq \left(\frac{\gamma(n+\beta)}{2}\right)^2 A+\frac{\gamma(n+\beta)(n-1)(n-2)}{2} B.
\end{align*}

\noindent \textbf{Step 2.} By the properties of $D_{k,h}$, it is easy to see that
\[
\int_M  \frac{u^2(x)}{{\rho^{4+\beta}_u}(x)}\cdot D_{k,h}(\rho_u(x))d\mathfrak{m}(x)\geq0
\]
when $k\leq 0$ and $h=0$.
Now we show that $\delta:=\frac{(n-4-\beta)^2(n+\beta)^2}{16}$ is sharp in this case. Obviously, it suffices to show that
\[
\delta=\inf_{u\in C^\infty_{0,F,d\mathfrak{m},\beta}(M)\backslash\{0\}}\frac{\int_M\frac{(\Delta u)^2}{\rho^\beta_u(x)} d\mathfrak{m}(x)}{\int_M \frac{u^2(x)}{{\rho^{4+\beta}_u}(x)}d\mathfrak{m}(x)}.
\]

Given $0<\epsilon<r<R<\widetilde{\mathfrak{i}}_p$, choose a cut-off function $h(t):[0,+\infty]\rightarrow [0,1]$ with $h|_{[0,r]}\equiv 1$, $h_{[R,+\infty)}\equiv0$ and $h'(t)\leq 0$. Now set $\psi(x):=h\circ\rho_-(x)$. Then $\psi|_{B^-_p(r)}\equiv1$ and $\text{supp}(\psi)\subset\overline{B^-_p(R)}$. Define $u_\epsilon(x):=\left[\max\{\epsilon,\rho_-(x)\}  \right]^{-\gamma}$. We set $u(x):=\psi(x)\cdot u_\epsilon(x)$. Clearly, Proposition \ref{nonempty} yields $G^\beta_{F,d\mathfrak{m}}(u)=0$.

Then we have
\begin{align*}
I_1(\epsilon)&=\int_M\frac{(\Delta u (x))^2}{\rho^\beta_u(x)}d\mathfrak{m}(x)\\
&=\int_{B^-_p(r)\backslash B^-_p(\epsilon)}\frac{(\Delta\rho^{-\gamma}_-(x))^2}{\rho^\beta_-(x)}d\mathfrak{m}(x)+\mathfrak{I}_2,\tag{4.6}\label{3.5}
\end{align*}
where
\[
\mathfrak{I}_2:=\int_{B^-_p(R)\backslash B^-_p(r)}\frac{(\Delta(\psi(x)\cdot \rho^{-\gamma}_-(x)))^2}{\rho^\beta_-(x)}d\mathfrak{m}(x)
\]
is finite and independent of $\epsilon$.
Since $\overline{B^-_p(r)}$ is compact, we can assume that $k_1\leq \mathbf{K}|_{\overline{B^-_p(r)}}\leq k\leq 0$. By the method used in Step 1 of Theorem \ref{Hardy1}, one can show that on $x\in \overline{B^-_p(r)}\backslash\widetilde{\text{Cut}}_p$,
\[
(n-1) \frac{\mathfrak{s}'_k(\widetilde{\rho}_+(x))}{\mathfrak{s}_k(\widetilde{\rho}_+(x))}  \leq \widetilde{\Delta}\widetilde{\rho}_+(x)\leq (n-1) \frac{\mathfrak{s}'_{k_1}(\widetilde{\rho}_+(x))}{\mathfrak{s}_{k_1}(\widetilde{\rho}_+(x))}.\tag{4.7}\label{4.7new}
\]
By (\ref{4.7new}), one gets
\begin{align*}
0\leq & -\gamma-1+ \widetilde{\rho}_+(x)\widetilde{\Delta}\widetilde{\rho}_+(x) \\
\leq& -\gamma-1+(n-1)D_{k_1,0}(\widetilde{\rho}_+(x))+(n-1)\\
\leq& -\gamma-1+{C}(n,k_1,r)+(n-1)\\
=&\frac{n+\beta}{2}+ {C}(n,k_1,r),\tag{4.8}\label{3.7}
\end{align*}
where ${C}(n,k_1,r)$ is a finite positive number only depending on $n$, $k_1$ and $r$ with $\lim_{r\rightarrow 0^+} {C}(n,k_1,r)=0$.
Hence, from (\ref{3.5}), (\ref{2.4}) and (\ref{3.7}), we obtain that
\[
I_1(\epsilon)\leq\gamma^2\left(\frac{n+\beta}{2}+C(n,k_1,r)\right)^2 \mathfrak{I}_1+\mathfrak{I}_2,\tag{4.9}\label{3.8}
\]
where
\begin{align*}
\mathfrak{I}_1:=\int_{{B}^-_p(r)\backslash {B}^-_p(\epsilon)} \widetilde{\rho}^{-n}_+(x)d\mathfrak{m}(x)=\int_{\widetilde{B}^+_p(r)\backslash \widetilde{B}^+_p(\epsilon)} \widetilde{\rho}^{-n}_+(x)d\mathfrak{m}(x).
\end{align*}
In particular, (\ref{2.8}) and (\ref{2.11}) imply that
\[
\lim_{\epsilon\rightarrow 0^+}\mathfrak{I}_1= +\infty.
\]
On the other hand, we have
\begin{align*}
I_2(\epsilon):=&\int_M \frac{u^2(x)}{{\rho^{4+\beta}_u}(x)}d\mathfrak{m}(x) \geq \int_{B^-_p(r)\backslash B^-_p(\epsilon)} \frac{(\psi u_\epsilon)^2(x)}{{\rho^{4+\beta}_-}(x)}d\mathfrak{m}(x)\\
=&\int_{B^-_p(r)\backslash B^-_p(\epsilon)} \frac{\rho^{-2\gamma}_-(x)}{{\rho^{4+\beta}_-}(x)}d\mathfrak{m}(x)=\mathfrak{I}_1.\tag{4.10}\label{3.9}
\end{align*}
Therefore, from (\ref{3.8}) and (\ref{3.9}), we obtain that
\begin{align*}
\delta&\leq\inf_{u\in C^\infty_{0,F,d\mathfrak{m},\beta}(M)\backslash\{0\}}\frac{\int_M\frac{(\Delta u)^2}{\rho^\beta_-(x)} d\mathfrak{m}(x)}{\int_M \frac{u^2(x)}{{\rho^{4+\beta}_-}(x)}d\mathfrak{m}(x)}\leq \lim_{\epsilon\rightarrow 0^+}\frac{I_1(\epsilon)}{I_2(\epsilon)}\\
&\leq \lim_{\epsilon\rightarrow 0^+}\frac{\gamma^2\left(\frac{n+\beta}{2}+{C}(n,k_1,r)\right)^2 \mathfrak{I}_1+\mathfrak{I}_2}{\mathfrak{I}_1}\\
&=\gamma^2\left(\frac{n+\beta}{2}+{C}(n,k_1,r)\right)^2\rightarrow \delta, \text{ as }r\rightarrow 0^+.
\end{align*}
Since $r$ is arbitrary, we are done.\end{proof}

\begin{remark}Under the same assumption as in Theorem \ref{Rll1}, one can obtain the following inequality by
an  argument similar to that of \cite[Theorem 3.2]{KR}
\begin{align*}
\int_M\frac{(\Delta u)^2}{\rho^\beta_u(x)} d\mathfrak{m}(x)&\geq \frac{(n+\beta)^2}{4}\int_M \frac{F^2(\nabla u)}{\rho^{\beta+2}_u(x)}d\mathfrak{m}(x)\\
& +\frac{(n-1)(n+\beta)(n-4-\beta)^2}{8}\int_M\frac{u^2(x)}{{\rho}^{\beta+4}_u(x)} D_{k,h}(\rho_u(x))d\mathfrak{m}(x),
\end{align*}
where the constant $\frac{(n+\beta)^2}{4}$ is sharp if $k\leq 0$ and $h=0$.  Moreover, one can show that this inequality is equivalent to the Rellich inequality in Theorem \ref{Rll1} by employing Theorem \ref{Hardy1}.
\end{remark}

Similar to Theorem \ref{Hardy2}, we can refine the Rellich inequality by adding terms like the Brezis-V\'azquez improvement. Before doing this, we need the following inequality, which is inspired by \cite[Lemma 4.5]{YSK}.

\begin{lemma}\label{de1}
Let $(M,F,d\mathfrak{m})$ be an $n$-dimensional  complete Finsler
manifold or an open domain containing $p$ with
 $\mathbf{K}\leq k<0$, $\mathbf{S}= 0$ and finite uniformity constant $\Lambda_F<\infty$.
 Thus, for any $\beta\in \mathbb{R}$ with $-2\leq \beta<n-4$ and any $u\in C^\infty_{0,F,d\mathfrak{m},\beta}(M)$, we have
\begin{align*}
&\int_M\rho^{-\beta}_u(x)\cdot \left[ \Delta u+\frac{(n+\beta)(n-\beta-4)}{4}\frac{u(x)}{\rho^2_u(x)}  \right]^2 d\mathfrak{m}(x)\\
\leq &\int_M \frac{(\Delta u)^2}{\rho^{\beta}_u(x)}d\mathfrak{m}(x)-\frac{(n+\beta)^2(n-\beta-4)^2}{16}\int_M\frac{u^2(x)}{\rho^{\beta+4}_u(x)}d\mathfrak{m}(x)\\
&-\frac{(n+\beta)(n-\beta-4)(n-1)(n-2)}{4}\int_M\frac{u^2(x)}{\rho_u^{\beta+4}(x)} D_{k,0}(\rho_u)d\mathfrak{m}(x)\\
&-\frac{(n+\beta)(n-\beta-4)}{2\Lambda_F}C(n,\Lambda_F,k)\int_M\frac{u^2(x)}{\rho^{\beta+2}_u(x)}d\mathfrak{m}(x),
\end{align*}
where $C(n,\Lambda_F,k)$ is a positive constant only depending on $\Lambda_F$ and $k$.

\end{lemma}
\begin{proof}
Given $u\in C^\infty_{0,F,d\mathfrak{m},\beta}(M)$, Theorem \ref{Hardy2} together with (\ref{*}) yields
\begin{align*}
&\int_M\frac{u\Delta u}{\rho^{\beta+2}_u(x)}d\mathfrak{m}(x)\\
=&\int_M\frac{\di(u\nabla u)}{\rho^{\beta+2}_u(x)}d\mathfrak{m}(x)-\int_M\frac{F^2(\nabla u)}{\rho^{\beta+2}_u(x)}d\mathfrak{m}(x)\\
=&-\frac12\int_M u^2(x)\cdot\varrho_{u,\beta}(x)\,d\mathfrak{m}(x)-\int_M\frac{F^2(\nabla u)}{\rho^{\beta+2}_u(x)}d\mathfrak{m}(x)\\
\leq&-\frac{(n-\beta-4)(n+\beta)}{4}\int_M \frac{u^2(x)}{{\rho}^{\beta+4}_u(x)}d\mathfrak{m}(x)-\frac{C(n,\Lambda_F,k)}{\Lambda_F}\int_M\frac{u^2(x)}{\rho^{\beta+2}_u(x)}d\mathfrak{m}(x)\\ &-\frac{(n-1)(n-2)}{2}\int_M\frac{u^2(x)}{\rho_u^{\beta+4}(x)} D_{k,0}(\rho_u)d\mathfrak{m}(x),\tag{4.11}\label{3.13}
\end{align*}
where $C(n,\Lambda_F,k)$ is a positive constant defined in Theorem \ref{Hardy2}.
Hence, we have
\begin{align*}
&\int_M\rho^{-\beta}_u(x)\cdot \left[ \Delta u+\frac{(n+\beta)(n-\beta-4)}{4}\frac{u(x)}{\rho^2_u(x)}  \right]^2 d\mathfrak{m}(x)\\
=&\int_M \frac{(\Delta u)^2}{\rho^{\beta}_u(x)}d\mathfrak{m}(x)+\frac{(n+\beta)^2(n-\beta-4)^2}{16}\int_M\frac{u^2(x)}{\rho^{\beta+4}_u(x)}d\mathfrak{m}(x)\\
&+\frac{(n+\beta)(n-\beta-4)}{2}\int_M\frac{u\Delta u}{\rho^{\beta+2}_u(x)}d\mathfrak{m}(x)
\end{align*}
\begin{align*}
\leq &\int_M \frac{(\Delta u)^2}{\rho^{\beta}_u(x)}d\mathfrak{m}(x)+\frac{(n+\beta)^2(n-\beta-4)^2}{16}\int_M\frac{u^2(x)}{\rho^{\beta+4}_u(x)}d\mathfrak{m}(x)+\frac{(n+\beta)(n-\beta-4)}{2}\\
&\times \left( -\frac{(n-\beta-4)(n+\beta)}{4}\int_M \frac{u^2(x)}{{\rho}^{\beta+4}_u(x)}d\mathfrak{m}(x)-\frac{(n-1)(n-2)}{2}\int_M\frac{u^2(x)}{\rho_u^{\beta+4}(x)} D_{k,0}(\rho_u)d\mathfrak{m}(x)\right.\\
&\left.-\frac{C(n,\Lambda_F,k)}{\Lambda_F}\int_M\frac{u^2(x)}{\rho^{\beta+2}_u(x)}d\mathfrak{m}(x)\right).
\end{align*}
Then the desired inequality follows.
\end{proof}

Now we have the following refined Rellich inequality. See \cite{KO,YSK} for more details in the Riemannian setting.
\begin{theorem}\label{complesRellich}
Let $(M,F,d\mathfrak{m})$ be an $n$-dimensional  complete Finsler
manifold or an open domain containing $p$ with
$\mathbf{K}\leq k<0$, $\mathbf{S}= 0$ and finite uniformity constant $\Lambda_F<\infty$.
Then there exists a positive constant $C=C(n,\Lambda_F,k)$ such that for any $\beta\in \mathbb{R}$ with $0\leq \beta<n-2$ and any $u\in C^\infty_{0,F,d\mathfrak{m},\beta}(M)$,
\begin{align*}
\int_M \frac{(\Delta u)^2}{\rho^{\beta}_u(x)}d\mathfrak{m}(x)\geq&\frac{(n+\beta)^2(n-\beta-4)^2}{16}\int_M\frac{u^2(x)}{\rho^{\beta+4}_u(x)}d\mathfrak{m}(x)\\
+&\frac{(n-1)(n-2)(n+\beta)(n-\beta-4)}{4}\int_M \frac{u^2(x)}{\rho^{\beta+4}_u(x)}D_{k,0}(\rho_u)d\mathfrak{m}(x)\\
+&\frac{(n-2-\beta)(n-2+\beta)C}{2\Lambda_F}\int_M\frac{u^2(x)}{\rho^{\beta+2}_u(x)}d\mathfrak{m}(x)\\
+&\frac{(n-1)(n-2)C}{\Lambda_F}\int_M \frac{u^2(x)}{\rho^{\beta+2}_u(x)}D_{k,0}(\rho_u)d\mathfrak{m}(x)+\frac{C^2}{\Lambda_F^2}\int_M\frac{u^2(x)}{\rho^{\beta}_u(x)}d\mathfrak{m}(x),
\end{align*}
where the constant $\frac{(n+\beta)^2(n-\beta-4)^2}{16}$ is sharp.
\end{theorem}

\begin{proof}From (\ref{3.13}), we  obtain that
\begin{align*}
&(\beta+1)\int_M\frac{u^2(x)}{\rho^{2+\beta}_u(x)}d\mathfrak{m}(x)+\frac{C(n,\Lambda_F,k)}{\Lambda_F}\int_M\frac{u^2(x)}{\rho^{\beta}_u(x)}d\mathfrak{m}(x)\\
&+\frac{(n-1)(n-2)}{2}\int_M \frac{u^2(x)}{\rho_u^{\beta+2}(x)}D_{k,0}(\rho_u)d\mathfrak{m}(x)\\
\leq &\int_M\frac{u(x)}{\rho^\beta_u(x)}\left( -\Delta u(x)-\frac{(n+\beta)(n-\beta-4)}{4}\frac{u(x)}{\rho^2_u(x)} \right)d\mathfrak{m}(x).\tag{4.12}\label{3.14}
\end{align*}
The inequalities $ab\leq\frac{\epsilon}{2}a^2+(2\epsilon)^{-1}b^2$ then yields
\begin{align*}
&\int_M\frac{u(x)}{\rho^\beta_u(x)}\left( -\Delta u(x)-\frac{(n+\beta)(n-\beta-4)}{4}\frac{u(x)}{\rho^2_u(x)} \right)d\mathfrak{m}(x)\\
\leq &\frac{\epsilon}{2}\int_M\frac{u^2(x)}{\rho^\beta_u(x)}d\mathfrak{m}(x)+\frac{1}{2\epsilon}\int_M\rho^{-\beta}_u(x)\cdot \left[ \Delta u(x)+\frac{(n+\beta)(n-\beta-4)}{4}\frac{u(x)}{\rho^2_u(x)}  \right]^2 d\mathfrak{m}(x),
\end{align*}
which together with (\ref{3.14}) implies that
\begin{align*}
&2\epsilon(\beta+1)\int_M\frac{u^2(x)}{\rho^{2+\beta}_u(x)}d\mathfrak{m}(x)+\left( \frac{2\epsilon C(n,\Lambda_F,k)}{\Lambda_F}-\epsilon^2\right)\int_M\frac{u^2(x)}{\rho^{\beta}_u(x)}d\mathfrak{m}(x)\\
&+\epsilon(n-1)(n-2)\int_M \frac{u^2(x)}{\rho_u^{\beta+2}(x)}D_{k,0}(\rho_u)d\mathfrak{m}(x)\\
\leq &\int_M\rho^{-\beta}_u(x)\cdot \left[ \Delta u+\frac{(n+\beta)(n-\beta-4)}{4}\frac{u(x)}{\rho^2_u(x)}  \right]^2 d\mathfrak{m}(x).
\end{align*}
By letting $\epsilon=C(n,\Lambda_F,k)/\Lambda_F$ and Lemma \ref{de1}, we obtain that
\begin{align*}
&2\frac{C(n,\Lambda_F,k)(\beta+1)}{\Lambda_F}\int_M\frac{u^2(x)}{\rho^{2+\beta}_u(x)}d\mathfrak{m}(x)+ \frac{C^2(n,\Lambda_F,k)}{\Lambda_F^2}\int_M\frac{u^2(x)}{\rho^{\beta}_u(x)}d\mathfrak{m}(x)\\
&+\frac{(n-1)(n-2)C(n,\Lambda_F,k)}{\Lambda_F}\int_M \frac{u^2(x)}{\rho_u^{\beta+2}(x)}D_{k,0}(\rho_u)d\mathfrak{m}(x)\\
\leq &\int_M\rho^{-\beta}_u(x)\cdot \left[ \Delta u+\frac{(n+\beta)(n-\beta-4)}{4}\frac{u(x)}{\rho^2_u(x)}  \right]^2 d\mathfrak{m}(x)\\
\leq &\int_M \frac{(\Delta u)^2}{\rho^{\beta}_u(x)}d\mathfrak{m}(x)-\frac{(n+\beta)^2(n-\beta-4)^2}{16}\int_M\frac{u^2(x)}{\rho^{4+\beta}_u(x)}d\mathfrak{m}(x)\\
&-\frac{(n+\beta)(n-\beta-4)(n-1)(n-2)}{4}\int_M \frac{u^2(x)}{\rho^{\beta+4}_u(x)}D_{k,0}(\rho_u)d\mathfrak{m}(x)\\
&-\frac{(n+\beta)(n-\beta-4)}{2\Lambda_F}C(n,\Lambda_F,k)\int_M\frac{u^2(x)}{\rho^{2+\beta}_u(x)}d\mathfrak{m}(x).
\end{align*}
\end{proof}

\begin{proof}[Proof of Theorem \ref{Th22}]It is not hard to see that the first consequence of Theorem \ref{Th22} follows from Theorem \ref{Rll1} while the second one is exactly Theorem \ref{complesRellich}.

\end{proof}

\section{Example}
 The aim of this section is to show the Hardy inequality in Theorem \ref{Hardy1} holds on the Randers space $(\mathbb{R}^n, F)$, where $F(x,y)=|y|+t\cdot y^n$, $y=(y^i)=y^i\frac{\partial}{\partial x^i}$ and $t\in [0,1)$.
Here, $(x^i)$ is the standard Euclidean coordinate system for $\mathbb{R}^n$.

It is not hard to see that $F$ is Berwalden (locally Minkowski) and hence,
$\mathbf{K}\equiv0$, $\mathbf{S}_{BH}\equiv0$, $\mathbf{S}_{HT}\equiv0$ and $\mathfrak{i}(M)=+\infty$, where $\mathbf{S}_{BH}$ (resp., $\mathbf{S}_{HT}$) is the S-curvature of the Busemann-Hausdorff (resp., the Holmes-Thompson) measure.
In particular, a geodesic in $(\mathbb{R}^n, F)$ is a straight line. Hence, a direct calculation yields
\[
\rho_+(x):=d_F(0,x)=|x|+t\cdot x^n,\ \rho_-(x):=d_F(0,x)=|x|-t\cdot x^n.\tag{5.1}\label{5.1}
\]
According to \cite[Example 3.2.1]{Sh1},
one can easily check that
\begin{align*}
F(\nabla \rho_+)=F^*(d \rho_+)=1, \ F(\nabla (-\rho_-))={F^*}(d(-\rho_-))=1.\tag{5.2}\label{new5.3}
\end{align*}

Now we show the following Hardy inequality by direct calculation.
\begin{proposition}
\[
\int_M  \frac{F^{2}(\nabla u(x))}{\rho^\beta_u(x)}d\mathfrak{m}(x)\geq\left(\frac{n-2-\beta}{2} \right)^2\int_M \frac{u^2(x)}{{\rho^{2+\beta}_u}(x)}d\mathfrak{m}(x),\ \forall\,u\in C^\infty_0(M).\tag{5.3}\label{5.3}
\]
where the constant $\left(\frac{n-2-\beta}{2} \right)^2$ is sharp. Here, $d\mathfrak{m}$ is either the Busemann-Hausdorff measure $d\mathfrak{m}_{BH}$ or the Holmes-Thompson measure $d\mathfrak{m}_{HT}$.
\end{proposition}
\begin{proof}
Since $F$ is a Randers metric, one gets (cf. \cite[Example 2.2.1]{Sh1})
\[
d\mathfrak{m}_{BH}=(1-t^2)^{\frac{n+1}2}dx,\ d\mathfrak{m}_{HT}=dx.
\]
So we just show (\ref{5.3}) under the Busemann-Hausdorff measure.

\noindent\textbf{Step 1.} As before, set $\Omega_1:=\{x\in M:\, u(x)>0\}$, $\Omega_2:=\{x\in M:\, u(x)<0\}$ and $\Omega_3:=\{x\in M:\, u(x)=0\}$. In order to obtain (\ref{5.3}), it suffices to prove that
\[
\int_{\Omega_i} \frac{F^{2}(\nabla u(x))}{\rho^\beta_u(x)}d\mathfrak{m}_{BH}(x)\geq\left(\frac{n-2-\beta}{2} \right)^2\int_{\Omega_i}  \frac{u^2(x)}{{\rho^{2+\beta}_u}(x)}d\mathfrak{m}_{BH}(x),\ 1\leq i\leq 3.\tag{5.4}\label{5.4}
\]

Obviously, (\ref{5.4}) is true when $i=3$. Now we consider the case of $i=1$. Firstly,
it follows from Example \ref{firstex} that
\[
F^{*2}(\xi+\eta)\geq F^{*2}(\xi)+2g^*_\xi(\xi,\eta),\ \forall \xi,\eta\in T^*M.\tag{5.5}\label{5.6}
\]
Now set $v=u\cdot{\rho_-}^\gamma$, where $\gamma=\frac{n-2-\beta}{2}$.
Let $\xi:=v\cdot\gamma {\rho_-}^{-\gamma-1}\cdot (-d{\rho_-})$ and $\eta:={\rho_-}^{-\gamma}\cdot dv$. Then it follows from (\ref{new5.3}) and (\ref{5.6}) that
\begin{align*}
&F^2(\nabla u)
\geq\gamma^2\cdot v^2\cdot{\rho_-}^{-2\gamma-2}+2\gamma\cdot{\rho_-}^{-2\gamma-1}\cdot v\cdot\langle \nabla ({-\rho_-}), dv\rangle
\end{align*}
holds on $\Omega_1-\{0\}$. Hence, we have
\begin{align*}
\int_{\Omega_1}  \frac{F^{2}(\nabla u)}{\rho^\beta_-(x)}d\mathfrak{m}_{BH}(x)&\geq \gamma^2\int_{\Omega_1}  \frac{u^2(x)}{\rho^{\beta+2}_-(x)} d\mathfrak{m}_{BH}(x)+R_0,\tag{5.6}\label{new 5.7}
\end{align*}
where
\begin{align*}
R_0:=&2\gamma \int_{\Omega_1}  {\rho_-}^{-2\gamma-1-\beta}(x)\cdot v(x)\cdot \langle \nabla (-{\rho_-}), dv\rangle\, d\mathfrak{m}_{BH}(x).
\end{align*}
In order to compute $R_0$, we  employ three coordinate transformations.
Firstly, set $X^\alpha:=x^\alpha$ for $1\leq \alpha \leq n-1$, and $X^n=\sqrt{1-t^2}\left( x^n-\frac{t \rho_-(x)}{1-t^2} \right)$.
Then (\ref{5.1}) yields that
\[
\frac{\partial X^n}{\partial x^n}=\frac{1-t\frac{\,x^n}{|x|}}{\sqrt{1-t^2}}>0,
\]
which implies $(X^i)$ is a coordinate system on $\mathbb{R}^n-\{0\}$.
In particular, one has
\begin{align*}
&|X|^2=\sum_{i=1}^n (X^i)^2=\frac{\rho_-^2(x) }{1-t^2},\tag{5.7}\label{5.7}\\
&dx^1\wedge \cdots \wedge dx^n=\frac{1}{\sqrt{1-t^2}}\left[ 1+\frac{tX^n}{|X|}  \right]dX^1\wedge\cdots \wedge dX^n.\tag{5.8}\label{5.8}
\end{align*}
Secondly,  we use the polar coordinates $(s,\theta^\alpha)\longleftrightarrow(X^1,\cdots,X^n)$, where $s=|X|$, $(\theta^\alpha)\in \mathbb{S}^{n-1}$.  It follows from (\ref{5.7}) and (\ref{5.8}) that
\begin{align*}
&\rho_-=s\sqrt{1-t^2},\tag{5.9}\label{5.9}\\
&d\mathfrak{m}_{BH}=(1-t^2)^{\frac{n}2}\left[1+t\cdot h(\theta^\alpha) \right]s^{n-1}ds\wedge d\nu_{\mathbb{S}^{n-1}},\tag{5.10}\label{5.10}
\end{align*}
where $h(\theta^\alpha):=X^n/|X|$ and $d\nu_{\mathbb{S}^{n-1}}$ is the area measure of $\mathbb{S}^{n-1}$.
Thirdly, set $\varrho:=-s\sqrt{1-t^2}$.  Then we have
$\varrho=-\rho_-$ and
\[
d\mathfrak{m}_{BH}=(-1)^n \left[1+t\cdot h(\theta^\alpha) \right]\varrho^{n-1}d\varrho\wedge d\nu_{\mathbb{S}^{n-1}}.\tag{5.11}\label{5.11}
\]

Now we claim that
\[
\frac{\partial}{\partial \varrho}=\nabla(-\rho_-).\tag{5.12}\label{5.12}
\]
In fact, it follows \cite[Example 3.2.1]{Sh1} and (\ref{5.1}) that
\[
\nabla(-\rho_-)(x_0)=-\frac{x_0^i\frac{\partial}{\partial x^i}}{|x|-tx^n}=-\frac{x_0}{\rho_-(x_0)}.\tag{5.13}\label{5.13}
\]
On the other hand, let $\gamma(\varsigma)$ be the unit speed (minimal) geodesic from $x_0$ to $0$. It is not hard to see that
\[
\gamma(\varsigma)=\left(1-\frac{\varsigma}{\rho_-(x_0)}  \right)x_0, \ \varsigma\in \left[0, \rho_-(x_0)\right],
\]
which together with (\ref{5.13}) yields
\[
\nabla(-\rho_-)(x_0)=\dot{\gamma}(0).
\]
Using (\ref{5.1}), one can check that $\gamma(\varsigma)$ is  a straight line from $\left( x_0^\alpha, \sqrt{1-t^2}\left(x^n_0-\frac{t\rho_-(x_0)}{1-t^2}\right) \right)$ to $0$
under the coordinate system $(X^i)$. Hence,
\begin{align*}
&\langle \nabla \varrho, d\theta^\alpha\rangle=\langle \nabla(-\rho_-), d\theta^\alpha\rangle=\left.\frac{d}{d\varsigma}\right|_{\varsigma=0}\theta^\alpha(\gamma(\varsigma))=0,\\
&\langle \nabla \varrho, d\varrho\rangle=\langle \nabla(-\rho_-), d(-\rho_-)\rangle=F^2(\nabla(-\rho_-))=1,
\end{align*}
which implies
\[
\nabla(-\rho_-)=\nabla \varrho=\frac{\partial}{\partial \varrho}.
\]
So the claim is true.

Under the coordinate system $(\varrho,\theta^\alpha)$,
we can suppose that
\[
\Omega_1\mapsto\{(\varrho,y): y\in \mathcal {S}\subset \mathbb{S}^{n-1},\, t^1_y< \varrho < t^2_y,\cdots, t^k_y< \varrho< t^{k+1}_y\}.
\]
Thus, $u(\varrho,y)>0$ for $t^1_y< \varrho < t^2_y,\cdots, t^k_y< \varrho< t^{k+1}_y$ and
$\lim_{\varrho\rightarrow t^i_y}v(\varrho,y)=0$.
Hence, (\ref{5.11}) together with (\ref{5.12}) yields
\begin{align*}
R_0
=(-1)^n\gamma \int_{\mathcal {S}}\left[1+t\cdot h(\theta^\alpha) \right] \left(\sum_{i}\int_{t^{i}_y}^{t^{i+1}_y}\frac{\partial v^2}{\partial \varrho} d \varrho\right)d\nu_{\mathbb{S}^{n-1}}=0,
\end{align*}
which together with (\ref{new 5.7}) implies that (\ref{5.4}) is true when $i=1$.
Similarly, one can show that (\ref{5.4}) holds for $i=2$.

\noindent\textbf{Step 2.} Now we show that (\ref{5.3}) is sharp.
Note that $\mathfrak{i}(M)=+\infty$. Therefore, it follows from (\ref{5.1}) that
\begin{align*}
B^-_0(R)=\left\{x:\,\rho_-(x)<R\right\},\ \forall R>0.\tag{5.14}\label{5.14}
\end{align*}
Now we  choose any $0<\epsilon<r<R$ and a cut-off function $\psi$ with $\text{supp}(\psi)=B^-_0(R)$ and $\psi|_{B^-_0(r)}\equiv1$. See \cite[Lemma 1, p.26]{CCL} for the construction of a cut-off function.
Set
\[
u_\epsilon(x):=\left[\max\left\{\epsilon,\rho_-(x)\right\}\right]^{-\gamma},\ u(x):=\psi(x)\cdot u_\epsilon(x),
\]
where $\gamma:=\frac{n-2-\beta}{2}$.
Then we set
\begin{align*}
I_1(\epsilon):=\int_M \frac{F^2(\nabla u)}{\rho^\beta_u(x)}d\mathfrak{m}_{BH}=:\gamma^2\mathfrak{I}_1+\mathfrak{I}_2,\tag{5.15}\label{5.15}
\end{align*}
where
\begin{align*}
\mathfrak{I}_1:=\int_{B^-_p(r)\backslash B^-_p(\epsilon)}{\rho^{-n}_-}d\mathfrak{m}_{BH},\ \mathfrak{I}_2:=\int_{B^-_p(R)\backslash B^-_p(r)}\frac{F^2(\nabla (\psi \rho^{-\gamma}_-))}{\rho^\beta_-(x)}d\mathfrak{m}_{BH}<\infty.
\end{align*}
It follows from (\ref{5.9}), (\ref{5.10}), (\ref{5.14}) and $\int_{\mathbb{S}^{n-1}}h\, d\nu_{\mathbb{S}^{n-1}}=0$ that
\begin{align*}
\mathfrak{I}_1=
(1-t^2)^{\frac{n}{2}}\int^{\frac{r}{\sqrt{1-t^2}}}_{\frac{\epsilon}{\sqrt{1-t^2}}}\frac{s^{n-1}}{(s\sqrt{1-t^2})^n}ds\int_{\mathbb{S}^{n-1}}d\nu_{\mathbb{S}^{n-1}}=n\omega_n\ln\left(\frac{r}{\epsilon} \right).
\end{align*}
Similarly, one has
\begin{align*}
I_2(\epsilon):=\int_M \frac{u^2(x)}{\rho^{2+\beta}_u(x)}d\mathfrak{m}_{BH}\geq \int_{B^-_0(r)\backslash B^-_0(\epsilon)}\rho^{-n}_-(x)\,d\mathfrak{m}_{BH}=n\omega_n\ln\left(\frac{r}{\epsilon} \right).
\end{align*}
Hence, it follows from (\ref{5.3}) and (\ref{5.15}) that
\[
\gamma^2\leq \inf_{u\in C^\infty_0(M)\backslash\{0\}}\frac{\int_M  \frac{F^{2}(\nabla u)}{\rho^\beta_u(x)}d\mathfrak{m}}{\int_M \frac{u^2(x)}{{\rho^{2+\beta}_u}(x)}d\mathfrak{m}}\leq \lim_{\epsilon\rightarrow 0^+}\frac{I_1(\epsilon)}{I_2(\epsilon)}=\gamma^2.
\]
\end{proof}

\section{Acknowledgements}
This work was supported by the National Natural Science Foundation of China (No. 11501202; No. 11471246; No. 11671352), and Shanghai NSF in China (No. 17ZR1420900).

\end{document}